\documentclass[11pt]{amsart}

\usepackage{epigamath}

\usepackage[english]{babel}

\numberwithin{equation}{section}

\usepackage[shortlabels,inline]{enumitem}
\setlist[enumerate,1]{label={\rm(\arabic*)}, ref={\rm\arabic*}}

\usepackage{amsmath, amssymb, amsfonts, latexsym, mathtools}
\usepackage{mathrsfs}
\usepackage{amsthm}
\usepackage{amscd}

\usepackage[all]{xy}

\theoremstyle{definition}
\newtheorem{definition}{Definition}[section]

\theoremstyle{plain}
\newtheorem{thm}[definition]{Theorem}
\newtheorem{prop}[definition]{Proposition}
\newtheorem{lemma}[definition]{Lemma}

\theoremstyle{remark}
\newtheorem{remark}[definition]{Remark}

\newcommand{\A}{\mathbf{A}}
\newcommand{\calF}{\mathcal{F}}
\newcommand{\rmF}{\mathrm{F}}
\newcommand{\Fp}{\mathbf{F}_p}

\newcommand{\rmH}{\mathrm{H}}
\newcommand{\upI}{\textup{I}}
\newcommand{\calK}{\mathcal{K}}
\newcommand{\frakm}{\mathfrak{m}}
\newcommand{\calO}{\mathcal{O}}
\newcommand{\frakp}{\mathfrak{p}}
\newcommand{\bfP}{\mathbf{P}}
\newcommand{\Q}{\mathbf{Q}}
\newcommand{\Z}{\mathbf{Z}}

\newcommand{\II}{\textup{I\hspace*{-1.2pt}I}}

\DeclareMathOperator{\CH}{CH}
\DeclareMathOperator{\Coker}{Coker}
\DeclareMathOperator{\Ext}{Ext}
\DeclareMathOperator{\Frac}{Frac}
\DeclareMathOperator{\Gal}{Gal}

\DeclareMathOperator{\Hom}{Hom}
\DeclareMathOperator{\Ker}{Ker}

\DeclareMathOperator{\Spec}{Spec}

\DeclareMathOperator{\Sw}{Sw}
\DeclareMathOperator{\Tor}{Tor}
\DeclareMathOperator{\ab}{ab}
\DeclareMathOperator{\ch}{char}

\DeclareMathOperator{\dt}{dt}

\DeclareMathOperator{\ord}{ord}

\DeclareMathOperator{\red}{red}
\DeclareMathOperator{\res}{res}
\DeclareMathOperator{\rsw}{rsw}

\DeclareMathOperator{\sw}{sw}

\DeclareMathOperator{\FCC}{FCC}
\DeclareMathOperator{\CC}{CC}
\DeclareMathOperator{\Mor}{Mor}

\newcommand{\set}[2]{\{#1\mid #2\}}
\newcommand{\map}[3]{#1\colon#2\to#3}

\newcommand{\lra}{\longrightarrow}

\newcommand{\supth}[1]{\ensuremath{#1^{\mathrm{th}}}}

\newcounter{foo}
\makeatletter
\newcommand{\otherlabel}[2]{\protected@edef\@currentlabel{#2}\label{#1}}
\makeatother
 
\EpigaVolumeYear{9}{2025} \EpigaArticleNr{20} \ReceivedOn{March 4, 2024}
\InFinalFormOn{November 19, 2024}
\AcceptedOn{April 22, 2025}
 
\title{F-characteristic cycle of a rank 1 sheaf \\ on an arithmetic surface}
\titlemark{F-characteristic cycle of a rank 1 sheaf on an arithmetic surface} 

\author{Ryosuke Ooe}
\address{Graduate School of Mathematical Sciences, University of Tokyo, 3-8-1 Komaba, Meguro-Ku, Tokyo 153-8914, Japan}
\email{ooe-ryosuke075@g.ecc.u-tokyo.ac.jp}

\authormark{R.~Ooe}

\AbstractInEnglish{We prove the rationality of the characteristic form for a degree~$1$ character of the Galois group of an abelian extension of henselian discrete valuation fields. We prove the integrality of the characteristic form for a rank~$1$ sheaf on a regular excellent scheme. These properties are shown by reducing to the corresponding properties of the refined Swan conductor proved by Kato. 

We define the F-characteristic cycle of a rank~$1$ sheaf on an arithmetic surface as a cycle on the FW-cotangent bundle using the characteristic form on the basis of the computation of the characteristic cycle in the equal-characteristic case by Yatagawa. The rationality and the integrality of the characteristic form are necessary for the definition of the F-characteristic cycle. 
We prove the intersection of the F-characteristic cycle with the 0-section computes the Swan conductor of cohomology of the generic fiber.}



\begin{document}



\maketitle

\begin{prelims}

\DisplayAbstractInEnglish

\bigskip

\DisplayKeyWords

\medskip

\DisplayMSCclass

\end{prelims}


\newpage

\setcounter{tocdepth}{1}

\tableofcontents


\section{Introduction}\label{sec:1}

Let $K$ be a henselian discrete valuation field with residue field $F$ of characteristic $p>0$, and let $L$ be a finite abelian extension of $K$. Kato \cite{Ka89} defined the refined Swan conductor of a character of the Galois group $\Gal(L/K)$ as an injection to the $F$-vector space $\Omega^1_F(\log)$. Recently, Saito \cite{Sa20} defined the characteristic form of such a character as a non-logarithmic variant of the refined Swan conductor. The characteristic form takes value in the $\overline{F}$-vector space $\rmH_1(L_{\overline{F}/\calO_K})$, where $\calO_K$ denotes the valuation ring of $K$ and $\rmH_1(L_{\overline{F}/\calO_K})$ denotes the first homology group of the cotangent complex. In the equal-characteristic case, the non-logarithmic theory played an important role in the computation of the characteristic cycle; see \cite[Section 7.3]{Sa17}. 

In Section~\ref{sec:5}, we show two properties of the characteristic form for rank~$1$ sheaves. The first property is the rationality of the characteristic form. 

\begin{thm}[Rationality, \textit{cf.} Theorem~\ref{2.2}]\label{0.2}
Let $\map{\chi}{\Gal(L/K)}{\Q/\Z}$ be a character. Let $m$ be the total dimension of $\chi$. Then the image of the characteristic form $\map{\ch\chi}{\frakm^m_{K}/\frakm^{m+1}_{K}}{\rmH_1(L_{\overline{F}/\calO_K})=\rmH_1(L_{F^{1/p}/\calO_K})\otimes_F\overline{F}}$ of $\chi$ is contained in $\rmH_1(L_{F^{1/p}/\calO_K})$. 
\end{thm}

The second property is the integrality of the characteristic form for rank~$1$ sheaves. Let $D$ be a divisor with simple normal crossings on a regular excellent scheme $X$.  Let $\{D_i\}_{i\in I}$ be the set of irreducible components of $D$, and let $K_i$ be the local field at the generic point $\frakp_i$ of $D_i$. Let $F_i$ be the residue field of  $K_i$. Let $U$ be the complement of $D$.  Let $\chi$ be an element of $\rmH^1(U,\Q/\Z)$. Let $Z_{\chi}$ be the union of the $D_i$ such that $\chi|_{K_i}$ is wildly ramified and $R_{\chi}=\sum_{i\in I}\dt(\chi|_{K_i})D_i$ be the total dimension divisor. We put $m_i=\dt(\chi|_{K_i})$. 

\begin{thm}[Integrality, \textit{cf.} Theorem~\ref{2.1} and \eqref{5.16}]\label{0.3}
There exists a unique global section $\ch(\chi)$ in $\Gamma(Z_{\chi}, F\Omega^1_X(pR_\chi)|_{Z_\chi})$ such that the germ at $\frakp_i$ is equal to the following composition of maps:
\[
\frakm_{K_i}^{m_i}/\frakm_{K_i}^{m_i+1}\xrightarrow{\ch(\chi|_{K_i})}\rmH_1(L_{F_i^{1/p}/\calO_{K_i}})\lra \rmH_1(L_{F_i^{1/p}/\calO_{K_i}})\otimes_{F_i^{1/p}} F_i.
\]
Here, the second map is induced by the \supth{p} power map $F_i^{1/p}\to F_i$.

\end{thm} 

In the case where the characteristic of $K$ is $p$, these properties have already been proved by using  Artin--Schreier--Witt theory by Matsuda \cite{Ma97} and Yatagawa \cite{Ya17}. In the case where the characteristic of $K$ is zero, Artin--Schreier--Witt theory does not work, so we need to use a different method. The strategy of the proofs of Theorems~\ref{0.2} and~\ref{0.3} is to reduce to the corresponding properties of the refined Swan conductor proved by Kato \cite{Ka89}.  To do this, we compare the refined Swan conductor with the characteristic form.

The relation between the refined Swan conductor and the characteristic form is explained as follows.  Let $\map{\chi}{\Gal(L/K)}{\Q/\Z}$ be a character. The characters are divided into two types. If $\chi$ is of type~\ref{I} (for example, the residue field extension is separable), the characteristic form of $\chi$ is the image of the refined Swan conductor of $\chi$. On the other hand, if $\chi$ is of type~\ref{II} (for example, the ramification index of $L/K$ is~$1$ and the residue field extension is inseparable), the refined Swan conductor of $\chi$ is the image of the characteristic form of $\chi$. A large part of the proof of these relations is due to Saito. The author thanks him for kindly suggesting the author to include the proof in this paper. 

For a character of type~\ref{I}, Theorem~\ref{0.2} holds since the characteristic form is the image of the refined Swan conductor and the refined Swan conductor takes value in the $F$-vector space $\Omega^1_F(\log)$. For a character of type~\ref{II}, we would like to change the character to a character of type~\ref{I}. The typical case where a character is of type~\ref{I} is when the residue field $F$ is perfect. Hence we would like to take an extension $K'$ of $K$ such that the residue field of $K'$ is perfect. In fact, it suffices to consider the field $K'$ with the $\supth{p}$ power roots of a lifting of a $p$-basis of $F$, though the residue field of $K'$ may not be perfect.

As in the proof of Theorem~\ref{0.2}, we prove Theorem~\ref{0.3} using the integrality of the refined Swan conductor, but the proof is more complicated. 

In Section~\ref{sec:6}, we consider the theory of the characteristic cycle.  The characteristic cycle of an \'etale sheaf on a smooth scheme over a perfect field of positive characteristic is defined by Saito \cite{Sa17}. The characteristic cycle is defined as a cycle on the cotangent bundle. By the index formula, the intersection with the 0-section computes the Euler characteristic if the scheme is projective.  The characteristic cycle was computed on a closed subset of codimension less than $2$ by using the characteristic form. Yatagawa \cite{Ya20} gave an explicit computation of the characteristic cycle of a rank~$1$ sheaf on a scheme of dimension~$2$.

The existence of the cotangent bundle on a scheme of mixed characteristic is not known. Instead, Saito \cite{Sa22-2} defined the FW-cotangent bundle $FT^*X|_{X_F}$ of a regular noetherian scheme $X$ over a discrete valuation ring $\calO_K$ of mixed characteristic $(0,p)$ to be the vector bundle of rank $\dim X$ on the closed fiber $X_F$ to consider the characteristic cycle of an \'etale sheaf on a scheme of mixed characteristic. The characteristic cycle in the mixed characteristic case has not been defined in general. 

Let $D$ be a divisor with simple normal crossings on $X$, and let $\map{j}{U=X-D}{X}$ be the open immersion. Let $\Lambda$ be a finite field of characteristic different from $p$, and let $\calF$ be a smooth sheaf of $\Lambda$-modules of rank~$1$.  In the case $\dim X=2$, we define the F-characteristic cycle $\FCC j_!\calF$ of $j_!\calF$ as a cycle on the FW-cotangent bundle on the basis of the computation in the equal-characteristic case by Yatagawa.

On a closed subset of codimension less than $2$, we define the F-characteristic cycle using the characteristic form.  To determine the coefficients of the fibers at closed points, we use both the refined Swan conductor and the characteristic form. The main reason for using both non-log and log theories is that after successive blowups, the refined Swan conductor becomes a locally split injection but the characteristic form has no such properties.  The rationality (Theorem~\ref{0.2}) and the integrality (Theorem~\ref{0.3}) of the characteristic form are crucial to determine the coefficients of the fibers.

In analogy with the index formula, we prove that the intersection of the F-characteristic cycle with the 0-section computes the Swan conductor of cohomology of the generic fiber.

\begin{thm}[Theorem~\ref{main}]\label{0.1}
Assume $\dim X=2$ and $X$ is proper over $\calO_K$. Then we have 
\[
\left(\FCC j_!\calF-\FCC j_!\Lambda, FT^*_XX|_{X_F}\right)_{FT^*X|_{X_F}}=p\cdot\left(\Sw_K\left(X_{\overline{K}}, j_!\calF\right)-\Sw_K\left(X_{\overline{K}}, j_!\Lambda\right)\right).
\]
\end{thm}

Abbes \cite{Ab00} found the formula computing the Swan conductor of cohomology of the generic fiber of an arithmetic surface under the assumption that a coefficient sheaf has no fierce ramification. Our formula restricts to a coefficient sheaf of rank~$1$ but needs no assumption on ramification. 

We prove Theorem~\ref{0.1} using Kato--Saito's conductor formula; see \cite{KS13}. We study the relation between the F-characteristic cycle and the pullback of the logarithmic characteristic cycle defined by Kato \cite{Ka94}. This step is similar to the computation by Yatagawa in the equal-characteristic case. 

We give an outline of the paper. In Section~\ref{sec:2}, we briefly recall the definition of the characteristic form and state the rationality and the integrality of the characteristic form explained above. In Section~\ref{sec:3}, we recall the definition and properties of the refined Swan conductor in parallel with the characteristic form. In Section~\ref{sec:4}, we give relations between the refined Swan conductor and the characteristic form. In Section~\ref{sec:5}, we prove the rationality and the integrality established in Section~\ref{sec:2} using the results in Section~\ref{sec:4}. In Section~\ref{sec:6}, we define the F-characteristic cycle of a rank~$1$ sheaf on an arithmetic surface. We prove the main theorem, which gives a formula computing the Swan conductor of cohomology of the generic fiber. We give an example of the F-characteristic cycle. 

\subsection*{Acknowledgments}
The author would like to express his sincere gratitude to his advisor Professor Takeshi Saito for suggesting the problem, giving a lot of helpful advice, and showing his unpublished book on ramification theory, which contains the contents of Section~\ref{sec:3} and the proof of Lemma~\ref{3.4} and Proposition~\ref{3.3}. The author thanks the anonymous referees for their careful reading and comments. 

\section{Characteristic form}\label{sec:2}
In this section, we recall the notion of characteristic form and state the integrality of the characteristic form.

\subsection{Cotangent complex and FW-differential}
We briefly recall the properties on cotangent complexes from \cite{Sa20}.  

Let $K$ be a discrete valuation field with valuation ring $\calO_K$ and with residue field $F$ of characteristic $p>0$. Let $E$ be a field containing $F$. 
For an element $u\in \calO_K$, we write $\overline{u}$ for the image of $u$ in $F$. If there exists a $\supth{p}$ root of $\overline{u}$ in $E$, the element $\tilde{d}u$ in $\rmH_1(L_{E/\calO_K})$ is defined in \cite[Equation~(1.9)]{Sa20}. We write $wu$ for this element instead of $\tilde{d}u$.

\begin{prop}\label{2.7}
Let $\pi$ be a uniformizer of\, $K$ and $(v_i)_{i\in I}$ be a $p$-basis of\, $F$.  Assume that the field $E$ contains $F^{1/p}$. Then, $\{w\pi, wv_i\}_{i\in I}$ forms a basis of the $E$-vector space $\rmH_1(L_{E/\calO_K})$.
\end{prop}

\begin{proof}
By \cite[Proposition 1.1.3(2)]{Sa20}, we have an exact sequence
\[
0\lra \frakm_K/\frakm_K^2\otimes_F E\overset{w}\lra \rmH_1(L_{E/\calO_K})\lra \Omega^1_F\otimes_F E\lra 0
\]
of $E$-vector spaces, where $\frakm_K$ denotes the maximal ideal of $\calO_K$. Then $\pi$ defines a basis of $\frakm_K/\frakm_K^2\otimes_F E$, and $\{dv_i\}_{i\in I}$ forms a basis of $\Omega^1_F\otimes_FE$. The assertion follows since the map $\rmH_1(L_{E/\calO_K})\to \Omega^1_F\otimes_F E$ sends $wv_i$ to $dv_i$ by
\cite[Proposition 1.1.4(2)]{Sa20}. 
\end{proof}

Let $L$ be a finite separable extension of $K$ with residue field $E$. The morphism $\Spec E\to \Spec\calO_L\to \Spec\calO_K$ of schemes defines the distinguished triangle
$L_{\calO_L/\calO_K}\otimes_{\calO_L}^{\mathrm{L}}E\to L_{E/\calO_K}\to L_{E/\calO_L}\to$ by \cite[Proposition II.2.1.2]{Il71}. 
Since we have quasi-isomorphisms $L_{E/\calO_L}\cong N_{E/\calO_L}[1]$ and $L_{\calO_L/\calO_K}\cong\Omega^1_{\calO_L/\calO_K}[0]$ by \cite[Lemma 1.2.6(4)]{Sa20}, we have an injection 
\begin{equation}\label{2.8}
\Tor_1^{\calO_L}(\Omega^1_{\calO_L/\calO_K}, E)\lra \rmH_1(L_{E/\calO_K})
\end{equation}
of $E$-modules.

The Frobenius--Witt differential was introduced by Saito \cite{Sa22-1} to define the cotangent bundle of a scheme over $\Z_{(p)}$. The following relation between the cotangent complex and the FW-differential is known. 

\begin{prop}[\textit{cf.} {\cite[Corollary 4.12]{Sa22-1}}]\label{2.3}
Let $A$ be a local ring with residue field $k$ of characteristic $p > 0$. Let $L_{k/A}$ denote the cotangent complex for the composition $\Spec k\xrightarrow{\rmF} \Spec k \to \Spec A$, where $\rmF$ is the Frobenius. Then, the canonical morphism $F\Omega^1_A\otimes_A k\to \rmH_1(L_{k/A})$ is an isomorphism.
\end{prop}

\subsection{Characteristic form}\label{s2.2}
We briefly recall the construction of the characteristic form in \cite{Sa20}. 
Let $K$ be a henselian discrete valuation field with residue field $F$ of characteristic $p>0$.  Let $G_K$ be the absolute Galois group of $K$, and let $(G^r_K)_{r\in \Q_{>0}}$ be Abbes--Saito's non-logarithmic upper ramification filtration; see \cite[Definition 3.4]{AS02}. For an element $\chi\in\rmH^1(G_K, \Q/\Z)$, we define the {\it total dimension} $\dt\chi$ to be the smallest rational number $r$ satisfying $\chi(G^{s}_K)=0$ for all $s>r$. The total dimension is an integer by \cite[Theorem 4.3.5]{Xi12} and \cite[Theorem 4.3.1]{Sa20}.

We fix some notation. Let $L$ be a finite separable extension of $K$, and let $K'$ be a separable extension of $K$ of ramification index $e$. Let $E,F'$ be the residue fields of $L, K'$, respectively. Let $S, S', T$ be the spectra of the valuation rings $\calO_K, \calO_{K'}, \calO_L$, respectively. Take a closed immersion $T\to P$ to a smooth scheme over~$S$. For a rational number $r>0$ such that $er$ is an integer, we define the scheme $P^{[r]}_{S'}$ to be the dilatation $P^{[D_{r}\cdot T_{S'}]}$ of $P_{S'}=P\times_{S}{S'}$ with respect to the Cartier divisor $D_r$ defined by $\frakm^{er}_{K'}$ and the closed subscheme $T_{S'}=T\times_SS'$. (See \cite[Definition 3.1.1]{Sa20} for the definition of the dilatation.) Let $P^{(r)}_{S'}$ be the normalization of $P^{[r]}_{S'}$. Let $P^{[r]}_{F'}$ and $P^{(r)}_{F'}$ be the closed fibers of $P^{[r]}_{S'}$ and $P^{(r)}_{S'}$, respectively.

For an immersion $T\to P$ to a smooth scheme over $S$, we have an exact sequence 
\begin{equation}\label{2.10}
0\lra N_{T/P}\lra \Omega^1_{P/S}\otimes_{\calO_P}\calO_T\lra \Omega^1_{T/S}\lra 0
\end{equation}
of $\calO_L$-modules. 
We say that an immersion $T\to P$ to a smooth scheme over $S$ is {\it minimal} if the map
\begin{equation}\label{2.6}
\Tor_1^{\calO_L}(\Omega^1_{T/S}, E)\lra N_{T/P}\otimes_{\calO_L}E
\end{equation}
induced by \eqref{2.10} is an isomorphism. There exists a minimal immersion by \cite[Lemma 1.2.3(1)]{Sa20}.

Let $L/K$ be a finite Galois extension, and let $G=\Gal(L/K)$ be the Galois group. Let $r>1$ be a rational number such that $G^{r+}=\cup_{s>r}G^s=1$. By the reduced fiber theorem, see \cite{BLR}, there exists a finite separable extension $K'$ of $K$ of ramification index $e$ such that $er$ is an integer and the geometric closed fiber $P_{S'}^{(r)}\times_{S'}\overline{F}$ is reduced, where $\overline{F}$ denotes an algebraic closure of $F'$. We define the scheme $\Theta_{L/K, F'}^{(r)}$ to be the vector bundle 
 $\Hom_{F}(\frakm^{er}_{K'}/\frakm^{er+1}_{K'}, \Tor_1^{\calO_L}(\Omega^1_{\calO_L/\calO_K}, E))^\vee$ over $\Spec (E\otimes_{F}F')_{\red}$. If we take a minimal immersion $T\to P$ to a smooth scheme over $S$, the isomorphism \eqref{2.6} induces an isomorphism $P_{F',\red}^{[r]}\to\Theta_{L/K, F'}^{(r)}$ by \cite[Proposition 3.1.3(2)]{Sa20}. We define the scheme $\Phi_{L/K, F'}^{(r)}$ to be $P_{F'}^{(r)}$. The definition does not depend on the choice of a minimal immersion $T\to P$, by \cite[Lemma 3.3.7]{Sa20}.  
 
We fix a morphism $\map{i_0}{L}{K_s}$ to a separable closure of $K$. Let $\overline{T}_{S'}$ be the normalization of $T\times_{S} S'$ and $\overline{T}_{\overline{F}}=\overline{T}_{S'}\times_{S'} \Spec \overline{F}$. Then the morphism $i_0$ can be regarded as a point of $\overline{T}_{\overline{F}}=\Mor_K(L, K_s)$. We have the cartesian diagram
\[
\xymatrix{
\overline{T}_{\overline{F}}\ar[r]\ar[d]& T_{\overline{F}}\ar[d]\\
 \Phi_{L/K,\overline{F}}^{(r)}\ar[r]& \Theta_{L/K,\overline{F}}^{(r)}
}
\]
by \cite[Lemma 3.3.7]{Sa20}.

Let $\Theta_{L/K, \overline{F}}^{(r)\circ}$ and $\Phi_{L/K,\overline{F}}^{(r)\circ}$ denote the connected component of $\Theta_{L/K, \overline{F}}^{(r)}$ and $\Phi_{L/K,\overline{F}}^{(r)}$, respectively,  containing the image of the closed point of $\overline{T}_{\overline{F}}$ corresponding to $i_0$. 
Then $\Phi_{L/K,\overline{F}}^{(r)\circ}$ is an additive $G^r$-torsor over $\Theta_{L/K, \overline{F}}^{(r)\circ}$ by \cite[Theorem 4.3.3(1)]{Sa20} in the sense of \cite[Definition 2.1.4(1)]{Sa20}. By \cite[Proposition 2.1.6]{Sa20}, there exists a group scheme structure on $\Phi_{L/K,\overline{F}}^{(r)\circ}$ such that 
\[
0\lra G^r\lra \Phi_{L/K,\overline{F}}^{(r)\circ}\lra \Theta_{L/K,\overline{F}}^{(r)\circ}\lra 0
\]
is an extension of smooth group schemes. We define the map  
\[
[\Phi]\colon\Hom(G^r,\Fp)\lra\rmH^1(\Theta_{L/K,\overline{F}}^{(r)\circ}, \Fp)
\]
sending a character $\chi$ to the image $\chi_*[\Phi_{L/K,\overline{F}}^{(r)\circ}]$ of $[\Phi_{L/K,\overline{F}}^{(r)\circ}]$ by $\map{\chi_*}{\rmH^1(\Theta_{L/K,\overline{F}}^{(r)\circ}, G^r)}{\rmH^1(\Theta_{L/K,\overline{F}}^{(r)\circ}, \Fp)}$. 
By \cite[Proposition 2.1.6]{Sa20}, the morphism $[\Phi]$ is an injection, and the image of $[\Phi]$ is contained in $\Ext(\Theta_{L/K,\overline{F}}^{(r)\circ}, \Fp)$.

Let $\frakm_{K_s}^r$ be the ideal $\set{x\in K_s}{\ord_Kx\ge r}$, and let $\frakm_{K_s}^{r+}$ be the ideal $\set{x\in K_s}{\ord_Kx> r}$ for the extension of the valuation of $K$ to $K_s$. 
If we identify $\Ext(\Theta_{L/K,\overline{F}}^{(r)\circ}, \Fp)$ with $({\Theta_{L/K,\overline{F}}^{(r)\circ}})^\vee=\Hom_{\overline{F}}(\frakm^{r}_{K_s}/\frakm^{r+}_{K_s}, \Tor_1^{\calO_L}(\Omega^1_{T/S}, \overline{F}))$ by the isomorphism \cite[Equation~(2..1)]{Sa20}, we have a commutative diagram 
\begin{equation}
\vcenter{
\xymatrix{
0\ar[r]& G^r\ar[r]\ar[d]^-{\chi}& \Phi_{L/K,\overline{F}}^{(r)\circ}\ar[r]\ar[d]& \Theta_{L/K,\overline{F}}^{(r)\circ}\ar[r]\ar[d]^-{[\Phi](\chi)}& 0\\
0\ar[r]& \Fp\ar[r]& \mathbf{G}_a\ar[r]& \mathbf{G}_a\ar[r]& 0}
}
\end{equation}
of extensions of smooth group schemes, where the lower extension is the Artin--Schreier extension. 

We define the {\it characteristic form} to be the composition of injections 
\[
\ch \colon \Hom\left(G^r, \Fp\right)\overset{[\Phi]}\lra \Hom_{\overline{F}}\left(\frakm^{r}_{K_s}/\frakm^{r+}_{K_s}, \Tor_1^{\calO_L}\left(\Omega^1_{T/S}, \overline{F}\right)\right)\lra \Hom_{\overline{F}}\left(\frakm^{r}_{K_s}/\frakm^{r+}_{K_s}, \rmH_1\left(L_{\overline{F}/\calO_K}\right)\right), 
\]
where the second morphism is induced by the injection \eqref{2.8}. 
For a character $\map{\chi}{G^r}{\Fp}$, we call $\ch\chi$ the {\it characteristic form} of $\chi$. 

We now state the rationality of the characteristic form. 

\begin{thm}\label{2.2}
Let $\chi\in \rmH^1(K,\Q/\Z)$ be a character of total dimension $m$. Then the image of the characteristic form $\map{\ch\chi}{\frakm^m_{K}/\frakm^{m+1}_{K}}{\rmH_1(L_{\overline{F}/\calO_K})=\rmH_1(L_{F^{1/p}/\calO_K}})\otimes_{F^{1/p}}\overline{F}$ of $\chi$ is contained in $\rmH_1(L_{F^{1/p}/\calO_K})$. 
\end{thm}
We give a proof of Theorem~\ref{2.2} in Section~\ref{sec:5}.

\begin{remark}\label{2.11}
We have an example where the image of the characteristic form is not contained in $\rmH_1(L_{{F/\calO_K}})$ when we assume that the characteristic of $F$ is 2. Consider the Kummer character $\chi$ defined by $t^2=1+\pi^{2(e-1)}u$, where $\pi$ is a uniformizer of $K$, $e=\ord_K2$ and $u\in \calO_K$ is such that $\sqrt{\overline{u}}\notin F$.
Then, the computation in \cite[Lemma 3.2.5.3]{Sa22-2} shows that we have  
\[
\ch\chi=\frac{wu-\sqrt{\overline{u}}\cdot w(2/\pi^{e-1})}{\pi^2} \in \rmH_1\left(L_{F^{1/2}/\calO_K}\right)\otimes_F \frakm_K^{-2}/\frakm_K^{-1}. 
\]

When the characteristic of the residue field is not 2, we can expect from the results in  equal characteristic, see \cite[Proposition 3.2.3]{Ma97} and \cite[Proposition 1.17]{Ya17}, that  the image of the characteristic form is contained in 
$\rmH_1(L_{F/\calO_K})$, but the author does not know how to prove. 
\end{remark}

We now state the integrality of the characteristic form. In this article, we define the local field of a ring $A$ at a prime ideal $\frakp$ as the fraction field of the completion of the localization $A_{\frakp}$. We note that the residue field may not be perfect. 

\begin{thm}\label{2.1}
Let $A$ be an excellent regular local ring of dimension $d$ with fraction field $K$ and with residue field $k$ of characteristic $p > 0$. We assume $c=[k:k^p]<\infty$ and fix a lifting $(x_l)_{l=1,\dots,c}$ of a p-basis of\, $k$ to $A$. Let $(\pi_i)_{i=1,\dots, d}$ be a regular system of parameters of\, $A$, and let $K_i$ be the local field at the prime ideal generated by $\pi_i$. 
We fix an integer $r$ satisfying $1\le r\le d$. 
Let $D_i$ be the divisor on $X=\Spec A$ defined by $\pi_i$, and let $U$ be the complement of $D=\cup_{i=1}^rD_i$. Let $\chi$ be an element of\, $\rmH^1(U, \Q/\Z)$, and let $\chi|_{K_i}$ be the pullback of $\chi$ by $\Spec K_i\to U$. We put $m_i = \dt(\chi|_{K_i})$. By Proposition $\ref{2.7}$ and Theorem $\ref{2.2}$, we may write
\[
\ch\left(\chi|_{K_j}\right)=\left(\sum_{1\le i\le d}\alpha_{i,j}w\pi_i+\sum_{1\le l\le c}\beta_{l,j}wx_l\right)/\pi_1^{m_1}\cdots \pi_r^{m_r}
\]
with $\alpha_{i,j},\beta_{l,j}\in {\Frac(A/\pi_j)}^{1/p}$ for $1\le i\le d$, $1\le j\le r$ satisfying $m_j \ge2$, and $1\le l \le c$. Then, the following properties hold:
\begin{enumerate}
\item We have $\alpha_{i,j}^p, \beta_{l,j}\in A/\pi_j$. 
\item For integers $j,j'$ satisfying $1\le j,j'\le r$, the images of $\alpha_{i,j}^p$ and $\alpha_{i,j'}^p$ in $A/(\pi_j)+(\pi_{j'})$ are equal for each $i$, and the images of $\beta_{l,j}$ and $\beta_{l,j'}$ in $A/(\pi_j)+(\pi_{j'})$ are equal for each $l$. 
\end{enumerate}
\end{thm}

We give a proof of Theorem~\ref{2.1} in Section~\ref{sec:5}.

\section{Refined Swan conductor}\label{sec:3}
In this section, we recall the notion of refined Swan conductor. The refined Swan conductor was defined by Kato \cite{Ka89} as an injection from the dual of the graded quotients to twisted cotangent spaces with logarithmic poles. Using Abbes--Saito's (logarithmic) ramification theory, see \cite{AS02}, Saito \cite{Sa12} defined another injection from the dual of the graded quotients to twisted cotangent spaces with logarithmic poles. The coincidence of these two notions of refined Swan conductor is verified by Kato and Saito; see \cite[Theorem 1.5]{KS19}. In this paper, we use the definition by Saito, but we slightly change the construction to compare with the characteristic form. The construction here is also given by Saito. 

The content of this section is based on Saito's unpublished book on
ramification theory.

\subsection{Some preliminaries}
Let $K$ be a discrete valuation field with valuation ring $\calO_K$ and with residue field $F$ of characteristic $p$. 
Let~$L$ be an extension of $K$ with valuation ring $\calO_L$ and with residue field $E$.

\begin{definition}\leavevmode  
\begin{enumerate}
\item We say that a scheme $Q$ locally of finite type over $S=\Spec \calO_K$ is log smooth over $S$ if the following conditions are satisfied:
\begin{itemize}
\item The scheme $Q$ is regular and flat over $S$, and the generic fiber $Q_K$ is smooth over $K$. 
\item The reduced closed fiber $D=Q_{F,\red}$ is smooth over $F$. 
\item For every point $x\in D$ where the multiplicity $m$ of $D$ in $Q_F$ is divisible $p$, there exist an open neighborhood $U$ of $x$ and a smooth morphism 
\[
U\lra \Spec \calO_K[x,u^{\pm 1}]/(ux^m-\pi)
\]
over $S$ where $\pi$ denotes a uniformizer of $\calO_K$. 
\end{itemize}

\item  Let $Q$ be a log smooth scheme over $S$. We say that an immersion $T=\Spec \calO_L\to Q$ over $S$ is exact if the inverse image $T\times_Q Q_{F,\red}$ is equal to $\Spec E$. 
\end{enumerate}
\end{definition}

Let $Q$ be a log smooth scheme over $S$. Let $D$ be the reduced closed fiber $Q_{F,\red}$. 
We introduce the following notation:
\[
\Omega^1_Q(\log)=\Omega^1_Q(\log D),
\]
\[
\Omega^1_{Q/S}(\log/\log)=\Coker\left(\Omega^1_S(\log)\otimes_{\calO_S}\calO_Q\lra \Omega^1_Q(\log)\right). 
\]
The $\calO_Q$-module $\Omega^1_Q(\log)$ is locally free. By checking the case  $Q=\Spec \calO_K[x,u^{\pm 1}]/(ux^m-\pi)$, we can also prove that $\Omega^1_S(\log)\otimes_{\calO_S}\calO_Q\to \Omega^1_Q(\log)$ is a locally splitting injection and $\Omega^1_{Q/S}(\log/\log)$ is a locally free $\calO_Q$-module.

For an exact immersion $T=\Spec \calO_L\to Q$ to a log smooth scheme over $S$, we have an exact sequence 
\begin{equation}\label{1.7}
0\lra N_{T/Q}\lra \Omega^1_{Q/S}(\log/\log)\otimes_{\calO_Q}\calO_L\lra \Omega^1_{T/S}(\log/\log)\lra 0
\end{equation}
of locally free $\calO_L$-modules. Here, the first arrow $N_{T/Q}\to \Omega^1_{Q/S}(\log/\log)\otimes_{\calO_Q}\calO_L$ is injective since the map $N_{T/Q}\otimes_{\calO_L} L\to \Omega^1_{Q/S}(\log/\log)\otimes_{\calO_Q} L$ of $L$-vector spaces is injective.

We say that an exact immersion $T\to Q$ to a log smooth scheme over $S$ is {\it minimal} if the map
\begin{equation}\label{1.6}
\Tor_1^{\calO_L}\left(\Omega^1_{T/S}(\log/\log), E\right)\lra N_{T/Q}\otimes_{\calO_L}E
\end{equation}
induced by \eqref{1.7} is an isomorphism.

\begin{lemma}[\textit{cf.} {\cite[Lemma 1.2.3(1)]{Sa20}}]\label{1.2}
There exists a minimal exact immersion $T\to Q$ to a log smooth scheme over $S$.
\end{lemma}

\begin{proof}
Let $\pi$ be a uniformizer of $K$ and $m$ be the ramification index of $L/K$. 
Take a system of generators $a_1,\dots, a_n\in \calO_L$ over $\calO_K$, and put $ua_1^m=\pi$ with $u\in \calO_L^{\times}$. We define an exact closed immersion 
\[
T=\Spec\calO_L\lra Q'=\Spec\calO_K[X_1,\dots, X_n, U^{\pm 1}]/(UX_1^m-\pi)
\]
to a log smooth scheme sending $X_1,\dots X_n, U$ to $a_1,\dots, a_n, u$. Let $I$ be the kernel of the map 
\[
\calO_K[X_1,\dots, X_n, U^{\pm 1}]/(UX_1^m-\pi)\lra \calO_L.
\]
Take a lifting $f_1,\dots, f_s\in I$ of a basis of the image of $N_{T/Q'}\otimes_{\calO_L}E\to\Omega^1_{Q'/S}(\log/\log)\otimes_{\calO_{Q'}}E$. Then the closed subscheme $Q$ of $Q'$ defined by the ideal $(f_1,\dots, f_s)$ is log smooth over $S$ on a neighborhood of $T$. We show that the immersion $T\to Q$ is minimal. 

The construction of $Q$ shows that $N_{T/Q}\otimes_{\calO_L} E\to \Omega^1_{Q/S}(\log/\log)\otimes_{\calO_Q}E$ is a zero map. Hence the exact sequence \eqref{1.7} induces the isomorphism $\Tor_1^{\calO_L}(\Omega^1_{T/S}(\log/\log), E)\to N_{T/Q}\otimes_{\calO_L}E$. 
\end{proof}

We fix an exact immersion $T=\Spec \calO_L\to Q$ to a log smooth scheme over $S$. 
We have the following commutative diagram: 
\[
\xymatrix{
0\ar[d]&&&0\ar[d]\\
\Tor_1^{\calO_L}\left(\Omega^1_{T/S}(\log/\log), E\right)\ar[d]\ar@{-->}[rrr]&&& \Omega^1_F(\log)\otimes_{F}E\ar[d]\\
N_{T/Q}\otimes_{\calO_Q} E\ar[r]^-{\cong}\ar[d]&N_{E/D}\ar[r]&\Omega^1_D\otimes_{\calO_D} E\ar[r]&\Omega^1_Q(\log)\otimes_{\calO_D} E \ar[d]\\
\Omega^1_{Q/S}(\log/\log)\otimes_{\calO_Q}E\ar@{=}[rrr]&&&\Omega^1_{Q/S}(\log/\log)\otimes_{\calO_Q}E\rlap{,}
}
\]
where the left vertical sequence is obtained from \eqref{1.7}. Since the vertical sequences are exact, we get the morphism
\begin{equation}\label{1.10}
\Tor_1^{\calO_L}\left(\Omega^1_{T/S}(\log/\log), E\right)\lra \Omega^1_F(\log)\otimes_{F}E.
\end{equation}
Let us prove that this morphism is independent of the choice of exact immersions $T\to Q$. If we take two exact immersions $T\to Q$, $T\to Q'$, we may assume there exists a morphism $Q\to Q'$ such that the diagram
\[
\xymatrix{
T\ar[r] \ar[rd]&Q\ar[d]\\
&Q'
}
\]
is commutative by replacing $Q$ by an etale neighborhood of $T$. Then the independence of the morphism \eqref{1.10} follows from the functoriality of $N_{T/Q}$, $\Omega^1_Q(\log)$ with respect to $Q$. 

\subsection{Refined Swan conductor}\label{s3.2}

Let $K$ be a henselian discrete valuation field with residue field $F$ of characteristic $p>0$.  Let $G_K$ be the absolute Galois group of $K$, and let $(G^r_{K,\log})_{r\in \Q_{>0}}$ be Abbes--Saito's logarithmic upper ramification filtration; see \cite[Definition 3.12]{AS02}. For an element $\chi\in\rmH^1(G_K, \Q/\Z)$, we define the {\it Swan conductor} $\sw\chi$ to be the smallest rational number $r$ satisfying $\chi(G^{s}_{K,\log})=0$ for all $s>r$. The Swan conductor is an integer since Kato's definition, see \cite{Ka89}, of the Swan conductor coincides with the definition here by \cite[Theorem~1.3]{KS19} and the Swan conductor is defined as an integer in Kato's definition. 

We use the same notation as in Section~\ref{s2.2}.  Take an exact closed immersion $T\to Q$ to a log smooth scheme over $S$. For a rational number $r>0$ such that $er$ is an integer, we define the scheme $Q^{[r]}_{S'}$ to be the dilatation $Q^{[D_{r}\cdot T_{S'}]}$ of $Q_{S'}$ with respect to the Cartier divisor $D_r$ defined by $\frakm^{er}_{K'}$ and the closed subscheme~$T_{S'}$. (See \cite[Definition 3.1.1]{Sa20} for the definition of the dilatation.)

Let $Q^{(r)}_{S'}$ be the normalization of $Q^{[r]}_{S'}$. Let $Q^{[r]}_{F'}$ and $Q^{(r)}_{F'}$ be the closed fibers of $Q^{[r]}_{S'}$ and $Q^{(r)}_{S'}$, respectively. 

Let $L/K$ be a finite Galois extension, and let $G=\Gal(L/K)$ be Galois group. Let $r>0$ be a rational number such that $G^{r+}_{\log}=\cup_{s>r}G^s_{\log}=1$. By the reduced fiber theorem, see \cite{BLR}, there exists a finite separable extension $K'$ of $K$ of ramification index $e$ such that $er$ is an integer and the geometric closed fiber $Q_{S'}^{(r)}\times_{S'}\overline{F}$ is reduced. We define the scheme $\Theta_{L/K, \log, F'}^{(r)}$ to be the vector bundle 
$\Hom_{F}(\frakm^{er}_{K'}/\frakm^{er+1}_{K'}, \Tor_1^{\calO_L}(\Omega^1_{\calO_L/\calO_K}(\log/\log), E))^\vee$ over $\Spec (E\otimes_{F}F')_{\red}$. If we take a minimal exact immersion $T\to Q$ to a log smooth scheme over $S$, the isomorphism \eqref{1.6} induces an isomorphism $Q_{F',\red}^{[r]}\to\Theta_{L/K,\log, F'}^{(r)}$ by \cite[Proposition 3.1.3(2)]{Sa20}. We define the scheme $\Phi_{L/K,\log, F'}^{(r)}$ to be $Q_{F'}^{(r)}$. By
a logarithmic variant of \cite[Lemma 3.3.7]{Sa20}, the definition of $\Phi_{L/K,\log, F'}^{(r)}$ does not depend on the choice of a minimal exact immersion $T\to Q$, and for every exact immersion $T=\Spec\calO_L\to Q$ to a log smooth scheme $Q$ over $S$, we have a cartesian diagram
\begin{equation}\label{1.9}
\vcenter{
\xymatrix{
\overline{T}_{\overline{F}}\ar[r]\ar[rd]&Q_{F'}^{(r)}\ar[r]\ar[d]&Q_{F',\red}^{[r]}\ar[d]\\
&\Phi_{L/K,\log, F'}^{(r)}\ar[r]&\Theta_{L/K,\log, F'}^{(r)}\rlap{.}\ar@{}[lu]|{\square}
}
}
\end{equation} 

We fix a morphism $\map{i_0}{L}{K_s}$ to a separable closure of $K$. Let $\Theta_{L/K, \log, \overline{F}}^{(r)\circ}$ and $\Phi_{L/K,\log, \overline{F}}^{(r)\circ}$ denote the connected components of $\Theta_{L/K,\log, \overline{F}}^{(r)}$ and $\Phi_{L/K,\log,\overline{F}}^{(r)}$, respectively, containing the image of the closed point of $\overline{T}_{\overline{F}}$ corresponding to $i_0$. 
 
\begin{prop}[\textit{cf.} {\cite[Theorem 4.3.3]{Sa20}}]
The $G^r_{\log}$-torsor $\Phi_{L/K,\log, \overline{F}}^{(r)\circ}$ over $\Theta_{L/K,\log, \overline{F}}^{(r)\circ}$ is additive. 
Hence, there exists a group scheme structure on $\Phi_{L/K,\log, \overline{F}}^{(r)\circ}$ such that the sequence
\[
0\lra G^r_{\log}\lra\Phi_{L/K,\log, \overline{F}}^{(r)\circ}\lra \Theta_{L/K, \log, \overline{F}}^{(r)\circ}\lra 0
\]
is an extension of smooth group schemes. 
\end{prop}

\begin{proof}[Proof \textup{(Saito)}]
First, we prove that $\Phi_{L/K,\log, \overline{F}}^{(r)\circ}$ is a $G^r_{\log}$-torsor over $\Theta_{L/K,\log, \overline{F}}^{(r)\circ}$. 
We write $0\in \Theta_{L/K,\log, \overline{F}}^{(r)\circ}$ for the image of the closed point of $\overline{T}_{\overline{F}}$ corresponding to $i_0$. Then the fiber $\Phi_{L/K,\log, \overline{F}}^{(r)\circ}\times_{\Theta_{L/K,\log, \overline{F}}^{(r)\circ}}0$ is identified with $\Phi_{L/K,\log, \overline{F}}^{(r)\circ}\cap \Mor_K(L,K_s)$. The latter is a $G^r_{\log}$-torsor, so the assertion follows.  

Next, we prove the $G^r_{\log}$-torsor is additive. We reduce the assertion to the case where the ramification index $e_{L/K}$ is $1$. By \cite[Theorem 3.1]{KS19}, there exists an extension $K'/K$ such that $e_{L'/K'}=1$ and the map $\Omega^1_F(\log)\to \Omega^1_{F'}(\log)$ is injective, where $L'=LK'$ denotes the composite field and $F'$ denotes the residue field of $K'$. From the functoriality of the construction of $\Phi_{\log}\to \Theta_{\log}$, we deduce the commutative diagram
\[
\xymatrix{
\Phi_{L'/K',\log, \overline{F'}}^{(r)\circ} \ar[r]\ar[d]&\Theta_{L'/K',\log, \overline{F'}}^{(r)\circ}\ar[d] \\
\Phi_{L/K,\log, \overline{F}}^{(r)\circ}\ar[r]& \Theta_{L/K,\log, \overline{F}}^{(r)\circ}\rlap{.}
}
\]
Hence, by \cite[Corollary 2.1.8(3)]{Sa20}, it suffices to show that the ${G'}^r_{\log}$-torsor $\Phi_{L'/K',\log, \overline{F}}^{(r)\circ}$ over $\Theta_{L'/K', \log, \overline{F}}^{(r)\circ}$ is additive, where $G'=\Gal(L'/K')$. 

If $e_{L/K}$ is 1, then we have $G^r=G^r_{\log}$ and the morphism $\Phi_{L/K,\log,\overline{F}}^{(r)\circ}\to \Theta_{L/K, \log,\overline{F}}^{(r)\circ}$ is equal to $\Phi_{L/K,\overline{F}}^{(r)\circ}\to \Theta_{L/K, \overline{F}}^{(r)\circ}$ since every immersion $T\to P$ to a smooth scheme is an exact immersion to the log smooth scheme $P$ under the assumption $e_{L/K}=1$. Hence the assertion follows from the fact that $\Phi_{L/K,\overline{F}}^{(r)\circ}$ is an additive torsor over $\Theta_{L/K, \overline{F}}^{(r)\circ}$; see \cite[Theorem 4.3.3(1)]{Sa20}. 
\end{proof}

In the same way as in Section~\ref{s2.2}, we get a morphism 
\[
[\Phi_{\log}]\colon \Hom_{\Fp}\left(G^r_{\log},\Fp\right)\lra \Hom_{\overline{F}}\left(\frakm^r_{K_s}/\frakm^{r+}_{K_s}, \Tor_1^{\calO_L}\left(\Omega^1_{T/S}(\log/\log), \overline{F}\right)\right). 
\]
We define the {\it refined Swan conductor} to be the composition of injections 
\begin{multline}\label{1.8}
\rsw \colon \Hom\left(G^r, \Fp\right)\xrightarrow{[\Phi_{\log}]}\Hom_{\overline{F}}\left(\frakm^r_{K_s}/\frakm^{r+}_{K_s}, \Tor_1^{\calO_L}\left(\Omega^1_{T/S}(\log/\log), \overline{F}\right)\right) \\
\lra\Hom_{\overline{F}}\left(\frakm^r_{K_s}/\frakm^{r+}_{K_s}, \Omega^1_{F}(\log)\otimes_F\overline{F}\right), 
\end{multline}
where the second morphism is induced by the map \eqref{1.10}. 
We call $\rsw\chi$ the {\it refined Swan conductor} of $\chi$ for $\map{\chi}{G^r}{\Fp}$. 

\begin{remark}
The construction of the refined Swan conductor here coincides with the construction in \cite{Sa12}. 
Indeed, using the notation in \cite{Sa12}, we have the  diagram
\[
\xymatrix{
0\ar[r]& G^r_{\log}\ar[r]\ar@{=}[d]& Q^{(r)}_{\overline{F}} \ar[r]\ar[d]& P_{\overline{F}}^{(r)}= Q^{[r]}_{\overline{F}}  \ar[r]\ar[d]& 0\\
0\ar[r]& G^r_{\log}\ar[r]&\Phi_{L/K,\log,\overline{F}}^{(r)\circ}\ar[r]& \Theta_{L/K,\log, \overline{F}}^{(r)\circ}\ar[r]& 0
}
\]
by \eqref{1.9}, where the right vertical map is induced by the second morphism of \eqref{1.8}. 
\end{remark}

\begin{prop}\label{1.3}
Let $\chi\in \rmH^1(K, \Q/\Z)$ be a character of Swan conductor $n$. Then, the image of the refined Swan conductor $\map{\rsw\chi}{\frakm^n_{K}/\frakm^{n+1}_{K}}{\Omega^1_F(\log)\otimes_F\overline{F}}$ is contained in $\Omega^1_F(\log)$.
\end{prop}

\begin{proof}
  The assertion follows from \cite[Theorem 1.5]{KS19} since the refined Swan conductor is defined by Kato
  as a map to $\Omega^1_F(\log)$. 
\end{proof}

We recall the integrality of the refined Swan conductor proved by Kato. 
 
\begin{thm}[\textit{cf.} {\cite[Theorem 7.1 and Proposition 7.3]{Ka89}}]\label{1.1}
Let $A$ be an excellent regular local ring of dimension~$d$ with fraction field $K$ and with residue field $k$ of characteristic $p > 0$.  We assume $c=[k:k^p]<\infty$ and fix a lifting $(x_l)_{l=1,\dots, c}$ of a $p$-basis of\, $k$ to $A$.  Let $(\pi_i)_{i=1,\dots, d}$ be a regular system of parameters of $A$, and let $K_i$ be the local field at the prime ideal generated by $\pi_i$.  We fix an integer $r$ satisfying $1\le r\le d$. Let $D_i$ be the divisor on $X=\Spec A$ defined by $\pi_i$, and let $U$ be the complement of\, $D=\cup_{i=1}^rD_i$.  Let $\chi$ be an element of\, $\rmH^1(U, \Q/\Z)$, and put $n_i=\sw(\chi|_{K_i})$. Write
\[
\rsw(\chi|_{K_j})=\left(\sum_{1\le i\le r}\alpha_{i,j}d\log \pi_i+\sum_{r+1\le i \le d}\alpha_{i,j}d\pi_i+\sum_{1\le l\le c}\beta_{l,j}dx_l\right)/\pi_1^{n_1}\cdots \pi_r^{n_r} 
\]
with $\alpha_{i,j}, \beta_{l,j}\in \Frac(A/\pi_j)$ for $1\le i\le d$, $1\le j\le r$ satisfying $n_j\ge 1$, and $1\le l \le c$. Then,  the following properties hold:
\begin{enumerate}
\item We have $\alpha_{i,j}, \beta_{l,j} \in A/\pi_j$. 
\item For integers $j,j'$ satisfying $1\le j,j'\le r$, the images of $\alpha_{i,j}$ and $\alpha_{i,j'}$ in $A/(\pi_j)+(\pi_{j'})$ are equal for each $i$ and the images of $\beta_{l,j}$ and $\beta_{l,j'}$ in $A/(\pi_j)+(\pi_{j'})$ are equal for each $l$.
\end{enumerate}
\end{thm}

\section{Comparison}\label{sec:4}

In this section, we compare the refined Swan conductor with the characteristic form. 

Let $K$ be a henselian discrete valuation field with residue field $F$ of characteristic $p>0$. Let $\chi$ be an element of $\rmH^1(K,\Q/\Z)$, and let $L$ be a finite abelian Galois extension of $K$ such that $\chi$ factors through $G=\Gal(L/K)$. Since we have $G^r\supset G^r_{\log}\supset G^{r+1}$ for a rational number $r>0$ by \cite[Lemma 5.3]{AS03} and the Swan conductor and the total dimension are integers (as explained in Sections~\ref{s2.2} and~\ref{s3.2}), we have $\dt(\chi)=\sw(\chi)+1$ or $\dt(\chi)=\sw(\chi)$. We say that $\chi$ is of 
\refstepcounter{foo}\otherlabel{I}{\rm{I}}{{\it type}~I}
if $\dt(\chi)=\sw(\chi)+1$ and $\chi$ is of
\refstepcounter{foo}\otherlabel{II}{\II}{\it type~\II}
if $\dt(\chi)=\sw(\chi)$. 
If the residue field $F$ of $K$ is perfect, the character $\chi$ is of type~\ref{I} by \cite[Proposition 6.3.1]{AS03}.

\begin{prop}[\textit{cf.} {\cite[Propositions 1.1.8(b) and 1.1.10]{Sa20}}]\label{3.1}
Let $K$ be a discrete valuation field with residue field $F$. There exists an extension $K'$ of\, $K$ with perfect residue field $F'$ such that
\begin{equation}\label{3.5}
\rmH_1\left(L_{\overline{F}/\calO_K}\right)\lra \rmH_1\left(L_{\overline{F'}/\calO_{K'}}\right)
\end{equation}
is injective and $e_{K'/K}$ is equal to~$1$. 
\end{prop}

\begin{prop}\label{3.2}
Let $K$ be a henselian discrete valuation field with residue field $F$. 
Let $L$ be a finite Galois extension of\, $K$ of Galois group $G$. Let $r>0$ be a rational number, and assume $G^{r+}_{\log}=1$ and $G^r_{\log}=G^{r+1}$. Then, there exists a commutative diagram
\[
   \xymatrix{
    \Hom_{\Fp}\left(G^r_{\log}, \Fp\right) \ar[r]^-{\rsw} \ar@{=}[d] & \Hom_{\overline{F}}\left(\frakm^r_{K_s}/\frakm^{r+}_{K_s}, \Omega^1_F(\log)\otimes_F\overline{F}\right) \ar[d] &  \\
   \Hom_{\Fp}\left(G^{r+1}, \Fp\right)\ar[r]^-{\ch} & \Hom_{\overline{F}}\left(\frakm^r_{K_s}/\frakm^{r+}_{K_s}, \rmH_1\left(L_{\overline{F}/\calO_K}\right)\otimes_{\calO_K}\frakm_{K}^{-1}\right)\rlap{,}
   }
\]
where the right vertical map is induced by the composition of the maps
\[
\Omega^1_F(\log)\otimes_F\overline{F}\xrightarrow{\res\otimes 1} \overline{F}\cong \overline{F}\otimes_{F}\frakm_{K}/\frakm^2_{K}\otimes_{\calO_K} \frakm_{K}^{-1}\ \overset{w}\lra \rmH_1(L_{\overline{F}/\calO_K})\otimes_{\calO_K}\frakm_{K}^{-1}.
\]
\end{prop}

We reduce the proof of Proposition~\ref{3.2} to the following case, which is proved by Saito. 

\begin{lemma}\label{3.4}
Proposition~{\rm\ref{3.2}} holds if the residue field $E$ of\, $L$ is a separable extension of\, $F$. \textup{(}In this case, the equality $G^r_{\log}=G^{r+1}$ holds by \textup{\cite[Proposition 6.3.1]{AS03})}. 
\end{lemma}

\begin{proof}[Proof \textup{(Saito)}]
  If $L/K$ is tamely ramified, then we have $G^r_{\log}=G^{r+1}=1$ and the assertion is trivial. Hence we may assume that $L/K$ is wildly ramified. Let $m=pn$ be the ramification index of $L/K$. Since $E$ is a separable extension of $F$ and thus $\calO_L$ is generated by a single element over $\calO_K$, we may take a minimal immersion $T=\Spec\calO_L\to P$ to a smooth scheme of relative dimension $1$ over $S=\Spec\calO_K$. We prove that the dilatation $P^{[1]}$ contains an open subscheme $Q$ such that $Q$ is log smooth over $S$ and the immersion $T\to Q$ is a minimal exact immersion. Let $x\in P$ the image of the closed point of $T$.  Since the assertion is local at $x$, after replacing $P$ by an open neighborhood of $x$, we may assume $P=\Spec A$ is affine and $\calO_L=A/f$ with $f\in A$. Let $\pi$ be a uniformizer of $K$, and let $s\in A$ be a lifting of a uniformizer of $L$. Further replacing $P$, we may assume that the canonical morphism $P\to \A^1_{S}=\Spec\calO_K[s]$ is \'etale. 
Let $\frakm_x=(f,s)$ be the maximal ideal of $A$ at $x$. Since $\pi$ is divisible by $s^m$ in $\calO_L=A/f$ and $\pi\in\frakm_x-\frakm_x^2$, we have $f\equiv \pi\mod\frakm_x^2$ and $\frakm_x=(\pi,s)$. We may write $f=a\pi+bs^m$ with $a,b\in A$. Since $f$ is not in $\frakm_x^2$, we see that $a$ is not in $\frakm_x$. Hence we may assume $a$ is a unit in $A$ by replacing $P$. We have $a\pi+bs^m=0\in \calO_L$, so $b$ is not in $\frakm_x$, and we may also assume $b$ is a unit. Then we have an equality $(f,\pi)=(\pi,s^m)$ of ideals of $A$. We have $P^{[1]}=\Spec A[s^m/\pi]=\Spec A[v]/(s^m-v\pi)$, and $P^{[1]}$ contains an open subscheme $Q=\Spec A[u^{\pm1}]/(us^m-\pi)=P\times_{\A^1_S}\calO_K[s,u^{\pm1}]/(us^m-\pi)$, which is log smooth over $S$ since $P\to \A^1_{S}$ is \'etale. Since the closed subscheme $Q_{F,\red}$ of $Q$ is defined by $s$, the inverse image $T\times_QQ_{F,\red}$ is $E$, and $T\to Q$ is an exact immersion. We note that $Q\to P$ induces an isomorphism $N_{T/P}\otimes_{\calO_K}\frakm_K^{-1}\to N_{T/Q}$. 

 Let $K'$ be a finite separable extension of $K$ such that the closed fibers of $P_{S'}^{(r)}$ and $Q_{S'}^{(r)}$ are reduced. By the functoriality of dilatations and normalizations, the middle square of the diagram
\begin{equation}
\vcenter{
   \xymatrix{
    \Phi_{L/K,F'}^{(r+1)} \ar[d] & P_{F'}^{(r+1)}\ar[l]_-{\cong} \ar[r]&P_{F',\red}^{[r+1]}\ar[r]^-{\cong}&\Theta_{L/K,F'}^{(r+1)}\ar[d]  \\
    \Phi_{L/K,\log,F'}^{(r)}& Q_{F'}^{(r)}\ar[l]_-{\cong} \ar[r]\ar[u]_-{\cong}&Q_{F',\red}^{[r]}\ar[r]^-{\cong}\ar[u]_-{\cong}&\Theta_{L/K,\log,F'}^{(r)}
      }
      }
\end{equation}
is commutative, and we have a commutative diagram 
\[
\xymatrix{
0\ar[r]& G^{r+1}\ar[r]\ar@{=}[d]& \Phi_{L/K, \overline{F}}^{(r+1)\circ} \ar[r]\ar[d]& \Theta_{L/K,\overline{F}}^{(r+1)\circ}\ar[r]\ar[d]& 0\\
0\ar[r]& G^r_{\log}\ar[r]&\Phi_{L/K,\log,\overline{F}}^{(r)\circ}\ar[r]& \Theta_{L/K,\log, \overline{F}}^{(r)\circ}\ar[r]& 0
}
\]
of extensions of smooth group schemes. Hence we have a commutative diagram
\[
   \xymatrix{
    \Hom_{\Fp}\left(G^{r}_{\log}, \Fp\right) \ar[r]^-{\rsw} \ar@{=}[d] & \Hom_{\overline{F}}\left(\frakm^r_{K_s}/\frakm^{r+}_{K_s},  \Tor_1^{\calO_L}\left(\Omega^1_{\calO_L/\calO_K}(\log/\log),\overline{F}\right)\right) \ar[d] &  \\
   \Hom_{\Fp}\left(G^{r+1}, \Fp\right)\ar[r]^-{\ch} & \Hom_{\overline{F}}\left(\frakm^r_{K_s}/\frakm^{r+}_{K_s},  \Tor_1^{\calO_L}\left(\Omega^1_{\calO_L/\calO_K},\overline{F}\right)\otimes_{\calO_K}\frakm_K^{-1}\right)\rlap{.}
   }
\]

It suffices to show that the diagram 
\begin{equation}\label{3.6}
\vcenter{
   \xymatrix{
    N_{T/Q}\otimes_{\calO_L} E\ar[r]\ar[d]_-{\cong}& \Omega^1_F(\log)\otimes_F E\ar[d] \\
         N_{T/P}\otimes_{\calO_L}E\otimes_{\calO_K}\frakm_K^{-1} \ar[r] &\rmH_1(L_{E/\calO_K})\otimes_{\calO_K}\frakm_K^{-1}
  }}
\end{equation}
is commutative. Since $g=f/\pi=a+bv$ defines a basis of $N_{T/Q}$, we consider the image of this basis. The left vertical map sends $g$ to $f\otimes \pi^{-1}$. The lower horizontal map sends $f\otimes \pi^{-1}$ to $wf\otimes \pi^{-1}=a\cdot w\pi\otimes \pi^{-1}$. The right horizontal map sends $g=f/\pi$ to $da+vdb+bvd\log v$. This is equal to $da+vdb+ad\log \pi$ in $\Omega^1_F(\log)\otimes E$ since we have $md\log s-d\log v-d\log\pi=0$ in $\Omega^1_Q(\log)$ and $p$ divides $m$ and $g=a+bv=0$ in $\calO_L=A[g]/g$. The right vertical map sends $da+vdb+ad\log\pi$ to $a\cdot w\pi\otimes \pi^{-1}$ since $a$ and $b$ are units in~$A$ and $da+vdb\in \Omega^1_E$. 
\end{proof}

\begin{proof}[Proof of Proposition~\ref{3.2}]
Let $K'$ be an extension as in Proposition~\ref{3.1}, and let $L'=LK'$ be the composition field and $G'=\Gal(L'/K')$ be the Galois group. Then, we have $G^{r+1}={G'}^{r+1}$ by \cite[Corollary 4.2.6]{Sa20}. Since the residue field $F'$ of $K'$ is perfect, we have ${G'}^r_{\log}={G'}^{r+1}$ by \cite[Proposition 6.3.1]{AS03}. Since we assume $G^r_{\log}=G^{r+1}$, we have $G^r_{\log}=G'^{r}_{\log}$. 

By the commutative diagram
\[
   \xymatrix{
 \Omega^1_F(\log)\otimes_F\overline{F} \ar[d]\ar[r] &  \Omega^1_{F'}(\log)\otimes_{F'}\overline{F'} \ar[d] &  \\
  \rmH_1\left(L_{\overline{F}/\calO_K}\right)\otimes_{\calO_K}\frakm_{K}^{-1}\ar[r] & \rmH_1\left(L_{\overline{F'}/\calO_{K'}}\right)\otimes_{\calO_{K'}}\frakm_{K'}^{-1}\rlap{,}
   }
\]
it suffices to show that the diagram
\[
\xymatrix{
    \Hom_{\Fp}\left(G^r_{\log}, \Fp\right) \ar[r]^-{\rsw} \ar@{=}[d] & \Hom_{\overline{F}}\left(\frakm^r_{K_s}/\frakm^{r+}_{K_s}, \Omega^1_F(\log)\otimes_F\overline{F}\right) \ar[d] &  \\
    \Hom_{\Fp}\left(G'^r_{\log}, \Fp\right) \ar[r]^-{\rsw} \ar@{=}[d] & \Hom_{\overline{F'}}\left(\frakm^r_{K'_s}/\frakm^{r+}_{K'_s}, \Omega^1_{F'}(\log)\otimes_{F'}\overline{F'}\right) \ar[d] &  \\
    \Hom_{\Fp}\left(G'^{r+1}, \Fp\right) \ar[r]^-{\ch} \ar@{=}[d] & \Hom_{\overline{F'}}\left(\frakm^r_{K'_s}/\frakm^{r+}_{K'_s}, \rmH_1\left(L_{\overline{F'}/\calO_{K'}}\right)\otimes_{\calO_{K'}}\frakm_{K'}^{-1}\right)  &  \\
   \Hom_{\Fp}\left(G^{r+1}, \Fp\right)\ar[r]^-{\ch} & \Hom_{\overline{F}}\left(\frakm^r_{K_s}/\frakm^{r+}_{K_s}, \rmH_1\left(L_{\overline{F}/\calO_K}\right)\otimes_{\calO_K}\frakm_{K}^{-1}\right)\ar[u]_{(\ref{3.5})}
   }
\]
is commutative since the map $\rmH_1(L_{\overline{F}/\calO_K})\to \rmH_1(L_{\overline{F'}/\calO_{K'}})$ is injective. The upper  and  lower squares are commutative by the functoriality of the refined Swan conductor and of the characteristic form, respectively. The middle square is commutative by Lemma~\ref{3.4} since $F'$ is perfect. 
\end{proof}

The following proposition is proved by Saito. 

\begin{prop}\label{3.3}
Let $K$ be a henselian discrete valuation field with residue field $F$. 
Let $L$ be a finite Galois extension of $K$ of Galois group $G$. Let $r>1$ be a rational number, and assume $G^{r+}=1$. Then, 
there exists a commutative diagram
\[
   \xymatrix{
     \Hom_{\Fp}
     \left(G^r, \Fp\right) \ar[r]^-{\ch} \ar[d] & \Hom_{\overline{F}}\left(\frakm^r_{K_s}/\frakm^{r+}_{K_s}, \rmH_1\left(L_{\overline{F}/\calO_K}\right)\right) \ar[d] &  \\
   \Hom_{\Fp}\left(G^r_{\log}, \Fp\right)\ar[r]^-{\rsw} & \Hom_{\overline{F}}\left(\frakm^r_{K_s}/\frakm^{r+}_{K_s}, \Omega^1_F(\log)\otimes_F\overline{F}\right)\rlap{,}
   }
\]
where the right vertical map is induced from the composition of the maps
\[
\rmH_1(L_{\overline{F}/\calO_K})\lra \Omega^1_F\otimes_{F}\overline{F}\lra \Omega^1_F(\log) \otimes_{F}\overline{F}.
\]
\end{prop}

\begin{proof}[Proof \textup{(Saito)}]
We show that there exists a commutative diagram
\begin{equation}\label{3.9}
\vcenter{
   \xymatrix{
  0\ar[r]&G^r_{\log}\ar[r]\ar[d]& \Phi_{L/K,\log,F'}^{(r)\circ} \ar[r] \ar[d] &\Theta_{L/K,\log,F'}^{(r)\circ}\ar[d]\ar[r]&0  \\
 0\ar[r]&G^r\ar[r]&\Phi_{L/K,F'}^{(r)\circ}\ar[r] & \Theta_{L/K,F'}^{(r)\circ}\ar[r]&0
      }}
\end{equation}
of extensions of smooth group schemes. 
We may take a minimal immersion $T\to P$ to a smooth scheme over $S=\Spec \calO_K$ and a minimal exact immersion $T\to Q$ to a log smooth scheme over $S$. By replacing $Q$ by an \'etale neighborhood, we may assume that there exists a morphism $Q\to P$. Let $K'$ be a finite separable extension of $K$ such that the closed fibers of $P_{S'}^{(r)}$ and $Q_{S'}^{(r)}$ are reduced. By the functoriality of dilatations and normalizations, we have a commutative diagram
\begin{equation}\label{3.7}
\vcenter{
   \xymatrix{
    Q_{S'}^{(r)} \ar[r] \ar[d] & Q_{S'}^{[r]}\ar[d] &  \\
   P_{S'}^{(r)}\ar[r] & P_{S'}^{[r]}\rlap{.}
      }}
\end{equation}
Since $Q\to P$ induces a morphism 
\begin{equation*}
\vcenter{
   \xymatrix{
    0\ar[r]&N_{T/P} \ar[r] \ar[d]& \Omega^1_{P/S}\otimes_{\calO_P}\calO_L \ar[r]\ar[d] &  \Omega^1_{T/S}\ar[r]\ar[d]& 0\\
    0\ar[r]&N_{T/Q} \ar[r] & \Omega^1_{Q/S}(\log/\log)\otimes_{\calO_Q}\calO_L \ar[r] &  \Omega^1_{T/S}(\log/\log)\ar[r]&0      
    }}
\end{equation*}
of free resolutions, we obtain a commutative diagram
\begin{equation}\label{3.8}
\vcenter{
   \xymatrix{
    \Tor_1^{\calO_L}\left(\Omega^1_{\calO_L/\calO_K},E\right) \ar[r]^-{\cong} \ar[d] & N_{T/P}\otimes_{\calO_L}E \ar[d]\\
   \Tor_1^{\calO_L}\left(\Omega^1_{\calO_L/\calO_K}(\log/\log),E\right)  \ar[r]^-{\cong} & N_{T/Q}\otimes_{\calO_L}E\rlap{,} 
            }}
\end{equation}
where the isomorphisms are \eqref{1.6} and \eqref{2.6}.
The diagram
\begin{equation}
\vcenter{
   \xymatrix{
    \Phi_{L/K,\log,F'}^{(r)} \ar[d] & Q_{F'}^{(r)}\ar[l]_-{\cong} \ar[r]\ar[d]&Q_{F',\red}^{[r]}\ar[d]\ar[r]^-{\cong}&\Theta_{L/K,\log,F'}^{(r)}\ar[d]  \\
    \Phi_{L/K,F'}^{(r)}& P_{F'}^{(r)}\ar[l]_-{\cong} \ar[r]&P_{F',\red}^{[r]}\ar[r]^-{\cong}&\Theta_{L/K,F'}^{(r)}
      }}
\end{equation}
is commutative by (\ref{3.7}), (\ref{3.8}) and the functoriality of normalizations and dilatations. 
Hence the diagram (\ref{3.9}) is commutative and defines a commutative diagram
\[
   \xymatrix{
   \Hom_{\Fp}\left(G^{r}, \Fp\right)  \ar[r]^-{[\Phi]} \ar[d] & \ar[d]  \Hom_{\overline{F}}\left(\frakm^r_{K_s}/\frakm^{r+}_{K_s},  \Tor_1^{\calO_L}\left(\Omega^1_{\calO_L/\calO_K},\overline{F}\right)\right)\\
  \Hom_{\Fp}\left(G^r_{\log}, \Fp\right) \ar[r]^-{[\Phi_{\log}]} &\Hom_{\overline{F}}\left(\frakm^r_{K_s}/\frakm^{r+}_{K_s}\right), \Tor_1^{\calO_L}\left(\Omega^1_{\calO_L/\calO_K}(\log/\log),\overline{F}\right)\rlap{.}
   }
\]
Hence it suffices to show that the diagram
\[
   \xymatrix{
    \Tor_1^{\calO_L}\left(\Omega^1_{\calO_L/\calO_K},E\right) \ar[r] \ar[d] & \rmH_1\left(L_{E/\calO_K}\right) \ar[d]  \\
   \Tor_1^{\calO_L}\left(\Omega^1_{\calO_L/\calO_K}(\log/\log),E\right)\ar[r] & \Omega^1_F(\log)\otimes_F E
   }
\]
is commutative, where $E$ denotes the residue field of $L$. We deduce from the injectivity of the map $\Omega^1_F(\log)\otimes_{F} E\to \Omega^1_Q(\log)\otimes_{\calO_Q} E$ and the commutative diagrams
\begin{equation*}
\vcenter{
   \xymatrix{
    \Tor_1^{\calO_L}\left(\Omega^1_{\calO_L/\calO_K},E\right) \ar[r] \ar[d] & N_{T/P}\otimes_{\calO_L}E \ar[d] &   \Tor_1^{\calO_L}(\Omega^1_{\calO_L/\calO_K}(\log/\log),E)  \ar[r] \ar[d] & N_{T/Q}\otimes_{\calO_L}E \ar[d]\\
   \rmH_1\left(L_{E/\calO_K}\right)\ar[r] & N_{E/P}\rlap{,} & \Omega^1_F(\log)\otimes_F E\ar[r] & \Omega^1_Q(\log)\otimes_{\calO_Q} E      }}
\end{equation*}
that it suffices to show that the diagram
\begin{equation*}
\vcenter{
   \xymatrix{
    \Tor_1^{\calO_L}\left(\Omega^1_{\calO_L/\calO_K},E\right) \ar[r] \ar[d] & N_{T/P}\otimes_{\calO_T}E \ar[r]\ar[d]&N_{E/P}\ar[d]
   &&\rmH_1\left(L_{E/\calO_K}\right)\ar[d]\ar[ll]\\
   \Tor_1^{\calO_L}\left(\Omega^1_{\calO_L/\calO_K}(\log/\log),E\right)  \ar[r] & N_{T/Q}\otimes_{\calO_L}E  \ar[r]& N_{E/Q}\ar[r]&\Omega^1_Q(\log)\otimes_{\calO_Q} E& \Omega^1_F(\log)\otimes_{F} E\ar[l]
            }
            }
\end{equation*}
is commutative. The left square is commutative by (\ref{3.8}). The middle square is commutative by the functoriality of conormal sheaves. The right square is commutative since 
the diagram
\[
\xymatrix{
\rmH_1\left(L_{E/\calO_{K}}\right)\ar[r]\ar[d]&\rmH_1\left(L_{E/F}\right)\ar[r]\ar[d]&\Omega^1_F\otimes_FE\ar[r]\ar[d]&\Omega^1_F(\log)\otimes_FE\ar[d]\\
N_{E/P}\ar[r]\ar[d]&N_{E/P_F}\ar[r]\ar[d]&\Omega^1_{P_F}\otimes_{\calO_{P_F}}E\ar[r]\ar[d]&\Omega^1_Q(\log)\otimes_FE\\
N_{E/Q}\ar[r]&N_{E/Q_{F,\red}}\ar[r]&\Omega^1_{Q_{F,\red}}\otimes_{\calO_{Q_{F,\red}}}E\ar[ur]&
}
\]
is commutative by the functoriality of cotangent complexes.
\end{proof}

\section{Proof of the rationality and the integrality}\label{sec:5}

In this section, we prove Theorems~\ref{2.2} and~\ref{2.1}. 

\begin{proof}[Proof of Theorem~\ref{2.2}]
Let $L$ be a finite abelian extension such that $\chi$ factors through $G=\Gal(L/K)$.  Let $\pi$ be a uniformizer, and let $(v_i)_{i\in I}$ be a family of elements of $\calO_K$ such that $(dv_i)_{i\in I}$ forms a basis of $\Omega^1_F$. We put $m=\dt(\chi)$. First we consider the case where the character $\chi$ is of type~\ref{I}. If we put
\[
\rsw(\chi)=\left(\alpha d\log \pi+\sum_{i\in I}\beta_idv_i\right)/\pi^{m-1}, 
\]
then we have 
\[
\ch(\chi)=(\alpha w\pi)/\pi^{m}
\]
by Proposition~\ref{3.2}, and the assertion follows from Proposition~\ref{1.3}. 

Second we consider the case where the character $\chi$ is of type~\ref{II}. If we put 
\[
\ch(\chi)=\left(\alpha w \pi+\sum_{i\in I}\beta_iwv_i\right)/\pi^m,
\]
then we have 
\[
\rsw(\chi)=\left(\sum_{i\in I}\beta_idv_i\right)/\pi^{m}
\]
by Proposition~\ref{3.3}. We see that the $\beta_i$ are contained in $F$ by Proposition~\ref{1.3}. We show that $\alpha$ is contained in $F^{1/p}$.  We define the discrete valuation ring $\calO_{K'}$ by
\[
\calO_{K'}=\calO_K[w_i]_{i\in I}/\left(w_i^p-v_i\right)_{i\in I}
\]
and let $K'$ be the fraction field of $\calO_{K'}$. Then the residue field $F'$ of $K'$ is $F^{1/p}$. Let $L'=LK'$ be the composite field and $G'=\Gal(L'/K')$ be the Galois group. 
The map $\rmH_1(L_{\overline{F}/\calO_{K}})\to\rmH_1(L_{\overline{F'}/\calO_{K'}})$ sends $w\pi$ to $w\pi$ and the other basis elements
to 0. The diagram 
\[
   \xymatrix{
    \Hom_{\Fp}\left(G^{m}, \Fp\right) \ar[r]^-{\ch} \ar[d] & \Hom_{\overline{F}}\left(\frakm^{m}_{K_{s}}/\frakm^{m+}_{K_{s}}, \rmH_1\left(L_{\overline{F}/\calO_{K}}\right)\right)\ar[d] &  \\
   \Hom_{\Fp}\left(G'^{m}, \Fp\right)\ar[r] & \Hom_{\overline{F'}}\left(\frakm^{m}_{K'_{s}}/\frakm^{m+}_{K'_{s}}, \rmH_1\left(L_{\overline{F'}/\calO_{K'}}\right)\right)
   }
\]
is commutative by the functoriality; see \cite[Equation~(4.17)]{Sa20}. Here, the lower horizontal arrow is $\ch$ if $G'^m\ne 1$ and zero if $G'^m=1$. If the coefficient of $w\pi$ is not zero, then the image of $\ch\chi$ by the right vertical arrow is not zero. Hence $G'^{m}$ is not trivial, and thus we have $\dt(\chi')=m$.  Let $\chi'$ be the image of the character $\chi$ by the left vertical arrow. Then the characteristic form $\ch(\chi')$ is of the form
\[
\ch(\chi')=\alpha\cdot w\pi/\pi^{m}. 
\]
If the character $\chi'$ is of type~\ref{II}, then the refined Swan conductor of $\chi'$ is zero and we have a contradiction. Hence the character $\chi'$ is of type~\ref{I}. By the first case, we have $\alpha 
\in {F'}=F^{1/p}$. 
\end{proof}

Next, we prove Theorem $\ref{2.1}$. We prepare the following lemma. 

\begin{lemma}\label{4.1}
We use the notation of Theorem~{\rm\ref{2.1}} and assume the dimension of $A$ is $2$. We define the regular local ring $A'$ of dimension $2$ by 
\[
A'=A[u_2,y_l]_{1\le l\le c}/\left(u_2^p-\pi_2, y_l^p-x_l\right)_{1\le l\le c}.
\]
The maximal ideal of $A'$ is generated by $\pi_1$ and $u_2$. 
Let $K'_1$ be the local field of $A'$ at the prime ideal generated by $\pi_1$, and let $K'_2$ be the local field of $A'$ at the prime ideal generated by $u_2$. Let $L_i$ be a finite abelian extension of $K_i$ $(i=1,2)$ such that $\chi|_{K_i}$ factors through $G_i=\Gal(L_i/K_i)$. Let $L'_i=L_iK'_i$ be the composite field, and put $G'_i=\Gal(L'_i/K'_i)$. Let $F_i$ and $F'_i$ be the residue fields of\, $K_i$ and $K'_i$, respectively. Let $U'$ be the pullback of\, $U$ by $\Spec A'\to \Spec A$, and let $\chi'\in \rmH^1(U',\Q/\Z)$ be the pullback of $\chi$. We put $m'_i=\dt(\chi'|_{K'_i})$. 

\begin{enumerate}
  \item\label{4.1-1} Assume that $\chi|_{K_1}$ is wildly ramified. If the coefficient of $w\pi_1$ in $\ch(\chi|_{K_1})$ is not zero, then we have $m'_1=m_1$ and the character $\chi'|_{K'_1}$ is of type~\ref{I}. If the coefficient of $w\pi_1$ in $\ch(\chi|_{K_1})$ is zero, then we have $m'_1<m_1$. 

  \item\label{4.1-2} Assume that $\chi|_{K_2}$ is of type~\ref{II}. If the coefficient of $w\pi_1$ in $\ch(\chi|_{K_2})$ is not zero, then we have $m'_2=pm_2$ and the character $\chi'|_{K'_2}$ is of type~\ref{II}. If the coefficient of $w\pi_1$ is zero, we have $\sw(\chi'|_{K'_2})<pm_2$. 
\end{enumerate}
\end{lemma}

\begin{proof}
  \eqref{4.1-1}~ The map $\rmH_1(L_{\overline{F_1}/\calO_{K_1}})\to\rmH_1(L_{\overline{F'_1}/\calO_{K'_1}})$ sends $w\pi_1$ to $w\pi_1$ and the other basis elements
  to 0. The diagram 
\[
   \xymatrix{
    \Hom_{\Fp}\left(G^{m_1}_1, \Fp\right) \ar[r]^-{\ch} \ar[d] & \Hom_{\overline{F_1}}\left(\frakm^{m_1}_{K_{1,s}}/\frakm^{m_1+}_{K_{1,s}}, \rmH_1\left(L_{\overline{F_1}/\calO_{K_1}}\right)\right)\ar[d] &  \\
    \Hom_{\Fp}\left(G'^{m_1}_1, \Fp\right)\ar[r] & \Hom_{\overline{F'_1}}\left(\frakm^{m_1}_{K'_{1,s}}/\frakm^{m_1+}_{K'_{1,s}}, \rmH_1\left(L_{\overline{F'_1}/\calO_{K'_1}}\right)\right)
   }
\]
is commutative by the functoriality, see \cite[Equation~(4.17)]{Sa20}, where the lower horizontal arrow is $\ch$ if $G'^{m_1}_1\ne 1$ and zero if $G'^{m_1}_1= 1$. The first assertion follows from the same argument as in the proof of Proposition~\ref{2.2}. If the coefficient of $w\pi_1$ in $\ch(\chi|_{K_1})$ is zero, then the image of $\ch\chi|_{K_1}$ by the right vertical arrow is zero. Hence we have $G'^{m_1}_1\subset \Ker\chi'|_{K'_1}$, and we have $m'_1<m_1$. 

\eqref{4.1-2}~ The map $\Omega^1_{F_2}(\log)\to\Omega^1_{F'_2}(\log)$ sends $d\log\pi_1$ to $d\log\pi_1$ and the other basis elements to 0. Since $\chi|_{K_2}$ is of type~\ref{II}, we have $\sw(\chi|_{K_2})=m_2$. The diagram 
\[
   \xymatrix{
    \Hom_{\Fp}\left(G_{2,\log}^{m_2}, \Fp\right) \ar[r]^-{\rsw} \ar[d] & \Hom_{\overline{F_2}}\left(\frakm^{m_2}_{K_{2,s}}/\frakm^{m_2+}_{K_{2,s}}, \Omega^1_{F_2}(\log)\otimes_{F_2}\overline{F_2}\right)\ar[d] &  \\
   \Hom_{\Fp}\left(G'^{pm_2}_{2,\log}, \Fp\right)\ar[r] & \Hom_{\overline{F'_2}}\left(\frakm^{pm_2}_{K'_{2,s}}/\frakm^{pm_2+}_{K'_{2,s}}, \Omega^1_{F'_2}(\log)\otimes_{F'_2}\overline{F'_2}\right)
   }
\]
is commutative by the functoriality, see \cite[Equation~(4.17)]{KS19}, where the lower horizontal arrow is $\rsw$ if $G'^{pm_2}_{2,\log}\ne 1$ and zero if $G'^{pm_2}_{2,\log}= 1$. Since we assume the coefficient of $w\pi_1$ is not zero, the coefficient of $d\log\pi_1$ is not zero and the image of $\rsw\chi|_{K_2}$ by the right vertical arrow is not zero. Hence $G'^{pm_2}_{2,\log}$ is not trivial, and thus we have $\sw(\chi'|_{K'_2})=pm_2$. Since the inequality $m'_2\le pm_2$ holds, we obtain $m'_2=pm_2$, and the character $\chi|_{K_2}$ is of type~\ref{II}. If the coefficient of $w\pi_1$ in $\ch(\chi|_{K_2})$ is zero, then the image of $\rsw\chi|_{K_2}$ by the right vertical arrow is zero. Hence we have $G'^{pm_2}_{2,\log}\subset \Ker \chi'|_{K'_2}$, and we have $\sw(\chi'|_{K'_2})<pm_2$. 
\end{proof}

\begin{proof}[Proof of Theorem~\ref{2.1}]
Since $A/\pi_j$ is regular, it suffices to show that $\alpha_{i,j}^p,\beta_{l,j}$ are elements of $(A/\pi_j)_{\mathfrak{q}}$ for any prime ideal $\mathfrak{q}$ of height~$1$ of $A/\pi_j$. By replacing $A$ by $A_{\mathfrak{q}}$, we may assume $\dim A=2$.  We use the notation of Lemma~{\rm\ref{4.1}}.

We divide the proof into six cases.
\begin{enumerate}[wide, label=(\alph*), ref=\alph*]
\item\label{case-a} \emph{The case where $r=1$ and $\chi|_{K_1}$ is of type~\ref{I}.}~ 
In this case, the characteristic form $\ch(\chi|_{K_1})$ is the image of the refined Swan conductor $\rsw(\chi|_{K_1})$ by Proposition~\ref{3.2}. If we put
\[
\rsw(\chi|_{K_1})=\left(\alpha_1d\log \pi_1+\alpha_2d\pi_2+\sum_{1\le l\le c}\beta_ldx_l\right)/\pi_1^{m_1-1},
\]
then we have
\[
\ch(\chi|_{K_1})=\alpha_1 w\pi_1/\pi_1^{m_1}.
\]
Since $\alpha_1$ is in $A/\pi_1$ by Theorem~\ref{1.1}, the assertion follows.

\item\label{case-b} \emph{The case where $r=1$ and $\chi|_{K_1}$ is of type~\ref{II}.}~
In this case, the refined Swan conductor $\rsw(\chi|_{K_1})$ is the image of the characteristic form $\ch(\chi|_{K_1})$ by Proposition~\ref{3.3}. 
Hence, if we put
\[
\ch(\chi|_{K_1})=\left(\alpha_1w \pi_1+\alpha_2w\pi_2+\sum_{1\le l\le c}\beta_lwx_l\right)/\pi_1^{m_1},
\]
then we have 
\[
\rsw(\chi|_{K_1})=\left(\alpha_2d\pi_2+\sum_{1\le l\le c}\beta_ldx_l\right)/\pi_1^{m_1}.
\]
This implies $\alpha_2, \beta_l\in A/\pi_1$ by Theorem~\ref{1.1}. 
It remains to prove $\alpha_1^p\in A/\pi_1$. If $\alpha_1=0$, then the assertion holds, so we may assume $\alpha_1$ is not $0$. Then we have $m'_1=m_1$ and 
\[
\ch(\chi'|_{K'_1})=\alpha_1\cdot w\pi_1/\pi_1^{m_1}, 
\]
and $\chi'|_{K'_1}$ is of type~\ref{I} by Lemma~\ref{4.1}.1. Hence we have $\alpha_1\in A'/\pi_1$ by case~\eqref{case-a} applied to the triple
$(A',U',\chi')$. Since $A/\pi_1$ is of characteristic $p$, we obtain $\alpha_1^p\in A/\pi_1$. 

\item\label{case-c} \emph{The case where $r=2$ and $\chi|_{K_1}$ or $\chi|_{K_2}$ is tamely ramified.}~ 
In this case, we can prove the assertion by a similar argument to that in cases~\eqref{case-a} and~\eqref{case-b}. 

\item\label{case-d} \emph{The case where $r=2$ and $\chi|_{K_1}$ and $\chi|_{K_2}$ are both of type~\ref{I}.}~ 
 If we put
\[
\rsw(\chi|_{K_1})=\left(\alpha_{1,1}d\log \pi_1+\alpha_{2,1}d\pi_2+\sum_{1\le l\le c}\beta_{l,1}dx_l\right)/\pi_1^{m_1-1}\pi_2^{m_2-1},
\]
\[
\rsw(\chi|_{K_2})=\left(\alpha_{1,2}d\pi_1+\alpha_{2,2}d\log \pi_2+\sum_{1\le l\le c}\beta_{l,2}dx_l\right)/\pi_1^{m_1-1}\pi_2^{m_2-1},
\]
then we have
\[
\ch(\chi|_{K_1})=\pi_2\alpha_{1,1}w \pi_1/\pi_1^{m_1}\pi_2^{m_2},
\]
\[
\ch(\chi|_{K_2})=\pi_1\alpha_{2,2}w\pi_2/\pi_1^{m_1}\pi_2^{m_2}
\]
by Proposition~\ref{3.2}. Since $\alpha_{1,1}$ is in $A/\pi_1$ and $\alpha_{2,2}$ is in $A/\pi_2$ by Theorem~\ref{1.1}, the assertion follows. We note that the coefficient of $w\pi_1$ in  $\ch(\chi|_{K_1})$ is contained in $\pi_2\cdot(A/\pi_1)$. 

\item\label{case-e} \emph{The case where $r=2$ and $\chi|_{K_1}$ is of type~\ref{II} and $\chi|_{K_2}$ is of type~\ref{I}, or $\chi|_{K_1}$ is of type~\ref{I} and $\chi|_{K_2}$ is of type~\ref{II}.}~
We only consider the case where $\chi|_{K_1}$ is of type~\ref{II} and $\chi|_{K_2}$ is of type~\ref{I}. If we put
\[
\ch(\chi|_{K_1})=\left(\alpha_{1,1}w \pi_1+\alpha_{2,1}w\pi_2+\sum_{1\le l\le c}\beta_{l,1}wx_l\right)/\pi_1^{m_1}\pi_2^{m_2},
\]
\[
\rsw(\chi|_{K_2})=\left(\alpha_{1,2}d\log\pi_1+\alpha_{2,2}d\log \pi_2+\sum_{1\le l\le c}\beta_{l,2}dx_l\right)/\pi_1^{m_1}\pi_2^{m_2-1},
\]
then we have
\[
\rsw(\chi|_{K_1})=\left(\alpha_{2,1}d\log \pi_2+\sum_{1\le l\le c}\pi_2^{-1}\beta_{l,1}dx_l\right)/\pi_1^{m_1}\pi_2^{m_2-1},
\]
\[
\ch(\chi|_{K_2})=\alpha_{2,2}w \pi_2/\pi_1^{m_1}\pi_2^{m_2}
\]
by Propositions~\ref{3.2} and~\ref{3.3}. By Theorem~\ref{1.1}, we have $\alpha_{2,1}, \beta_{l,1}\in A/\pi_1$ and $\alpha_{2,2}\in A/\pi_2$ and the equalities $\alpha_{2,1}=\alpha_{2,2}$ and $\beta_{l,1}=0$ in $A/(\pi_1)+(\pi_2)$.  Hence it suffices to show that $\alpha_{1,1}^p\in A/\pi_1$ and $\alpha_{1,1}^p=0$ in $A/(\pi_1)+(\pi_2)$. 

If $\alpha_{1,1}=0$, then the assertion holds, so we may assume $\alpha_{1,1}$ is not $0$. Then the characteristic form $\ch(\chi'|_{K'_1})$ is of the form 
\[
\ch(\chi'|_{K'_1})=\alpha_{1,1}\cdot w\pi_1/\pi_1^{m_1}u_2^{pm_2}=u_2^{m'_2-pm_2}\alpha_{1,1}\cdot w\pi_1/\pi_1^{m_1}u_2^{m'_2}, 
\]
and $\chi'|_{K'_1}$ is of type~\ref{I} by Lemma~\ref{4.1}.1. Since we assume $\chi|_{K_2}$ is of type~\ref{I}, we have $\sw(\chi|_{K_2})=m_2-1$. Since the ramification index of the extension $K'_2/K_2$ is $p$, we have $\sw(\chi'|_{K'_2})\le p(m_2-1)$ by \cite[Proposition 5.1.1]{KS19}. Thus we have $m'_2-pm_2<0$. 
Hence we have $\alpha_{1,1}\in u_2\cdot(A'/\pi_1)$ by case~\eqref{case-c} applied to the pair
$(A',U',\chi')$ if $\chi'|_{K'_2}$ is tamely ramified, by case~\eqref{case-d}  if $\chi'|_{K'_2}$ is of type~\ref{I}, and by the first half of the argument in case~\eqref{case-e} if $\chi'|_{K'_2}$ is of type~\ref{II}. 
Hence we obtain $\alpha_{1,1}^p\in \pi_2\cdot(A/\pi_1)$. 

\item\label{case-f} \emph{The case where $r=2$ and $\chi|_{K_1}$ and $\chi|_{K_2}$ are both of type~\ref{II}.}~
If we put
\[
\ch(\chi|_{K_1})=\left(\alpha_{1,1}w \pi_1+\alpha_{2,1}w\pi_2+\sum_{1\le l\le c}\beta_{l,1}wx_l\right)/\pi_1^{m_1}\pi_2^{m_2},
\]
\[
\ch(\chi|_{K_2})=\left(\alpha_{1,2}w \pi_1+\alpha_{2,2}w\pi_2+\sum_{1\le l\le c}\beta_{l,2}wx_l\right)/\pi_1^{m_1}\pi_2^{m_2},
\]
then we have
\[
\rsw(\chi|_{K_1})=\left(\pi_2\alpha_{2,1}d\log\pi_2+\sum_{1\le l\le c}\beta_{l,1}dx_l\right)/\pi_1^{m_1}\pi_2^{m_2},
\]
\[
\rsw(\chi|_{K_2})=\left(\pi_1\alpha_{1,2}d\log \pi_1+\sum_{1\le l\le c}\beta_{l,2}dx_l\right)/\pi_1^{m_1}\pi_2^{m_2}
\]
by Proposition~\ref{3.3}. By Theorem~\ref{1.1}, we have $\beta_{l,1}=\beta_{l,2}$ in $A/(\pi_1)+(\pi_2)$. It suffices to show $\alpha_{1,1}^p\in A/\pi_1$,  $\alpha_{1,2}^p\in A/\pi_2$ and $\alpha_{1,1}^p=\alpha_{1,2}^p\in A/(\pi_1)+(\pi_2)$ since the assertion corresponding to $\alpha_{2,1}$ and $\alpha_{2,2}$ is proved by switching $\pi_1$ and $\pi_2$.

If $\alpha_{1,1}\ne 0$ and $\alpha_{1,2}\ne 0$, then we have $\dt(\chi'|_{K'_1})=m_1$ and $\dt(\chi'|_{K'_2})=pm_2$ by Lemma~\ref{4.1}. Further, $\ch(\chi'|_{K'_1})$ is of type~\ref{I} and $\ch(\chi'|_{K'_2})$ is of type~\ref{II},  and we have
\[
\ch(\chi'|_{K'_1})=\alpha_{1,1}w \pi_1/\pi_1^{m_1}u_2^{pm_2},
\]
\[
\ch(\chi'|_{K'_2})=\alpha_{1,2}w \pi_1/\pi_1^{m_1}u_2^{pm_2}
\]
by Lemma~\ref{4.1}. By  case~\eqref{case-e}, we have $\alpha_{1,1}\in A'/\pi_1$, $\alpha_{1,2}\in A'/u_2$ and $\alpha_{1,1}=\alpha_{1,2}\in A'/(\pi_1)+(u_2)$. Hence we obtain $\alpha_{1,1}^p\in A/\pi_1, \alpha_{1,2}^p\in A/\pi_2$ and $\alpha_{1,1}^p=\alpha_{1,2}^p\in A/(\pi_1)+(\pi_2)$.

If $\alpha_{1,1}\ne 0$ and $\alpha_{1,2}=0$, then we have 
\[
\ch(\chi'|_{K'_1})=\alpha_{1,1}w \pi_1/\pi_1^{m_1}u_2^{pm_2}=u_2^{m'_2-pm_2}\alpha_{1,1}w\pi_1/\pi_1^{m_1}u_2^{m'_2}
\]
and $\chi'|_{K'_1}$ is of type~\ref{I} by Lemma~\ref{4.1}\eqref{4.1-1}.  
If $\chi'|_{K'_2}$ is tamely ramified, then we have $u_2^{m'_2-pm_2}\alpha_{1,1}\in A'/\pi_1$ by  case~\eqref{case-c} and $m'_2-pm_2=1-pm_2<0$. 
If $\chi'|_{K'_2}$ is of type~\ref{I}, then we have $u_2^{m'_2-pm_2}\alpha_{1,1}\in u_2\cdot(A'/\pi_1)$ by the last note in  case~\eqref{case-d} and $m'_2-pm_2=1+\sw(\chi'|_{K'_2})-pm_2\le 0$ by Lemma~\ref{4.1}\eqref{4.1-2}. 
If $\chi'|_{K'_2}$ is of type~\ref{II}, then we have $u_2^{m'_2-pm_2}\alpha_{1,1}\in A'/\pi_1$ by case~\eqref{case-e} and $m'_2-pm_2=\sw(\chi'|_{K'_2})-pm_2< 0$ by Lemma~\ref{4.1}\eqref{4.1-2}. 
Hence we have $\alpha_{1,1}\in u_2\cdot(A'/\pi_1)$ in any case, and we obtain $\alpha_{1,1}^p\in \pi_2\cdot(A/\pi_1)$. 

If $\alpha_{1,1}=0$ and $\alpha_{1,2}\ne 0$, then we prove $\alpha_{1,2}^p\in \pi_1\cdot(A/\pi_2)$ by  induction on $m_1=\dt(\chi|_{K_1})>1$. 
By Lemma~\ref{4.1}\eqref{4.1-1}, we have $m'_1<m_1$, and by Lemma~\ref{4.1}\eqref{4.1-2}, we have 
\[
\ch(\chi'|_{K'_2})=\alpha_{1,2}w \pi_1/\pi_1^{m_1}u_2^{pm_2}=\pi_1^{m'_1-m_1}\alpha_{1,2}w \pi_1/\pi_1^{m'_1}u_2^{pm_2}, 
\]
and $\chi'|_{K'_2}$ is of type~\ref{II}.  
If $\chi'|_{K'_1}$ is tamely ramified or of type~\ref{I}, the assertion is true by  case~\eqref{case-c}  or~\eqref{case-e}, respectively. If $\chi'|_{K'_1}$ is of type~\ref{II},  we have $\pi_1^{p(m'_1-m_1)}\alpha_{1,2}^p\in \pi_1\cdot(A'/u_2)$ by the induction hypothesis. Hence we have $\alpha_{1,2}\in \pi_1\cdot(A'/u_2)$, and we obtain $\alpha_{1,2}^p\in \pi_1\cdot(A/\pi_2)$.  \qedhere 
\end{enumerate}
\end{proof}

\section{F-characteristic cycle}\label{sec:6}
In this section, we define the F-characteristic cycle of a rank~$1$ sheaf on a regular  surface as a cycle on the FW-cotangent bundle. We prove that the intersection with the 0-section computes the Swan conductor of cohomology. We give an example of the F-characteristic cycle. 

\subsection{Refined Swan conductor and characteristic form of a rank~1 sheaf}
Let $K$ be a discrete valuation field of characteristic 0 with residue field $F$ of characteristic $p > 0$. Let $X$ be a regular flat separated scheme of finite type over the valuation ring $\calO_K$ of $K$, and let $D$ be a divisor with simple normal crossings. 
Let $\{D_i\}_{i\in I}$ be the irreducible components of $D$, and let $K_i$ be the local field at the generic point $\frakp_i$ of $D_i$. 
Let $U$ be the complement of $D$. 
Let $\chi$ be an element of $\rmH^1(U,\Q/\Z)$. 
We define the Swan conductor divisor $R_{\chi}$ of $\chi$ by 
\[
R_\chi=\sum_{i\in I} \sw(\chi|_{K_i})D_i
\]
and denote the support of $R_\chi$ by $Z_\chi$. We note that $Z_{\chi}$ is contained in the closed fiber of $X$. Indeed, if $D_i$ intersects the generic fiber of $X$, the character $\chi|_{K_i}$ is tamely ramified since the characteristic of $K$ is zero.
 
By Theorem~\ref{1.1}, there exists a unique global section
\[
\rsw(\chi)\in \Gamma\left(Z_{\chi}, \Omega^1_X(\log D)(R_{\chi})|_{Z_{\chi}}\right)
\]
such that the germ $\rsw(\chi)_{\frakp_i}$ of $\rsw(\chi)$ is $\rsw(\chi|_{K_i})$ if the character $\chi|_{K_i}$ is wildly ramified. We call $\rsw(\chi)$ the {\it refined Swan conductor} of $\chi$.

\begin{definition}[\textit{cf.} {\cite[Definition 4.2]{Ka94}}]
Let $x$ be a closed point of $Z_\chi$. For $i\in I$ satisfying $x\in D_i\subset Z_\chi$, we define $\ord(\chi;x,D_i)$ to be the maximal integer $n\ge 0$ such that
\[
\rsw(\chi)|_{D_i,x}\in \frakm_x^n\Omega^1_X(\log D)(R_\chi)|_{D_i,x}, 
\]
where $m_x$ is the maximal ideal of $\calO_{X,x}$. We say that $(X,U,\chi)$ is \textup{clean} at $x$ if the integer $\ord(\chi;x,D_i)$ is zero for every $i\in I$ satisfying $x\in D_i\subset Z_\chi$. We say that $(X,U,\chi)$ is \textup{clean} if $(X,U,\chi)$ is clean at every closed point in $Z_\chi$. 
\end{definition}

We define the total dimension divisor $R'_{\chi}$ by
\[
R'_{\chi}=\sum_{i\in I}\dt\left(\chi|_{K_i}\right)D_i.
\]
By Proposition~\ref{2.3} and Theorem~\ref{2.1}, there exists a unique global section
\begin{equation}\label{5.16}
\ch(\chi)\in \Gamma\left(Z_{\chi}, F\Omega^1_X\left(pR'_{\chi}\right)|_{Z_{\chi}}\right)
\end{equation}
such that the germ $\ch(\chi)_{\frakp_j}$ of $\ch(\chi)$ is 
\[
\left(\sum_{1\le i\le d}\alpha_{i,j}^p w\pi_i+\sum_{1\le l\le c}\beta_{l,j}^pwx_l\right)/\pi_1^{pm_1}\cdots \pi_r^{pm_r}
\]
using the notation of Theorem~\ref{2.1} if the character $\chi|_{K_j}$ is wildly ramified. We call $\ch(\chi)$ the {\it characteristic form} of $\chi$.

\begin{definition}\label{5.18}
Let $x$ be a closed point of $Z_\chi$. For $i\in I$ satisfying $x\in D_i\subset Z_\chi$, we define $n'$ to be the maximal integer $n'\ge 0$ such that
\[
\ch(\chi)|_{D_i,x}\in \frakm_x^{n'}F\Omega^1_X(pR'_\chi)|_{D_i,x}, 
\]
where $m_x$ is the maximal ideal of $\calO_{X,x}$. We define $\ord'(\chi;x,D_i)$ by $\ord'(\chi;x,D_i)=n'/p$. We say that $(X,U,\chi)$ is \textup{non-degenerate} at $x$ if $\ord'(\chi;x,D_i)$ is zero for every $i\in I$ satisfying $x\in D_i\subset Z_\chi$, and we say that $(X,U,\chi)$ is \textup{non-degenerate} if $(X,U,\chi)$ is non-degenerate at every point at $x \in Z_{\chi}$.
\end{definition}

\begin{remark}
By  definition, $p\cdot \ord'(\chi;x,D_i)$ is an integer,  but $\ord'(\chi;x,D_i)$ may not be an integer. Assume that the characteristic of the residue field of $K$ is 2, and put $e=\ord_K2$. We consider the scheme ${X=\Spec\calO_K[T, (1+\pi^{2(e-1)}T^3)^{-1}]}$. Let $U$ be the  generic fiber $\Spec K[T, (1+\pi^{2(e-1)}T^3)^{-1}]$,  and let ${\chi\in \rmH^1(U,\mathbf{F}_2)}$ be the Kummer character defined by $t^2=1+\pi^{2(e-1)}T^3$. Then we have 
\[
\ch(\chi)=\frac{T^4\cdot wT-T^3\cdot w(2/\pi^{e-1})}{\pi^4} 
\]
and  $\ord'(\chi, x, X_F)=3/2$,  where $x$ denotes the closed point defined by $(\pi, T)$. 

Similarly to Remark~\ref{2.11}, we can expect that $\ord'(\chi;x, D_i)$ is an integer if the characteristic of the residue field of $K$ is not 2. 
\end{remark}

\subsection{F-characteristic cycle}
Let $K$ be a complete discrete valuation field of characteristic 0 with perfect residue field $F$ of characteristic $p > 0$. Let $X$ be a regular flat separated scheme of finite type over the valuation ring $\calO_K$ of $K$, and let $D$ be a divisor with simple normal crossings. 
We assume that $X$ is purely of dimension~$2$. 
Let $D_1,\dots, D_n$ be the irreducible components of $D$,  and let $K_i$ be the local field at the generic point $\frakp_i$ of $D_i$. We put $U=X-D$ and let $\map{j}{U}{X}$ be the open immersion. Let $X_F$ and $D_F$ be the closed fibers of $X$ and $D$. We fix a finite field $\Lambda$ of characteristic $l\ne p$. Let $\calF$ be a locally constant constructible sheaf of $\Lambda$-modules of rank 1 on $U$, and let $\map{\chi}{\pi_1^{\ab}(U)}{\Lambda^\times}$ be the corresponding character. We fix an inclusion $\Lambda^\times \to \Q/\Z$ and regard $\chi$ as an element of $\rmH^1(U,\Q/\Z)$. 

Let $I_{T,\chi}, I_{W,\chi}, I_{\textup{I},\chi},I_{\II,\chi}$ be the subsets of $I$ consisting of the $i\in I$ such that $\chi|_{K_i}$ is tamely ramified, wildly ramified, of type~\ref{I} and of type~\ref{II}, respectively. For a closed point $x$ in $D$, let $I_x$ be the subset of $I$ consisting of $i\in I$ such that $x\in D_i$ and $I_{*,\chi, x}$ be $I_{*,\chi}\cap I_x$, where $*=W,T,\textup{I},\II$. 
Let $Z_{\II,\chi}$ be the union $\cup_{i\in I_{\II,\chi}}D_i$. 

We define the sub--vector bundle $L_{i,\chi}$ of $T^*X(\log D)|_{D_i}$ for $i\in I_{W,\chi}$ as the image of the multiplication by the refined Swan conductor of $\chi$, 
\[
\times \rsw(\chi)|_{D_i}\colon \calO_X\left(-R_{\chi}\right)\otimes_{\calO_X}\calO_{D_i}\left(\sum_{x\in D_i}\ord(\chi;x,D_i)[x]\right)\lra\Omega^1_X(\log D)|_{D_i}. 
\]

\begin{definition}[\textit{cf.} {\cite[Equation~(3.4.4)]{Ka94}}]
Assume that $(X,U,\chi)$ is clean. We define the \textup{logarithmic characteristic cycle} $\CC^{\log}j_!\calF$ as a cycle on the logarithmic cotangent bundle $T^*X(\log D)|_{D_F}$ by 
\[
\CC^{\log}j_!\calF=\left[T^*_XX(\log D)|_{D_F}\right]+\sum_{i\in I_{W,\chi}}\sw\left(\chi|_{K_i}\right)\left[L_{i,\chi}\right], 
\]
where $T^*_XX(\log D)|_{D_F}$ denotes the \textup{0}-section of $T^*X(\log D)|_{D_F}$. 
\end{definition}

In the case $\dim X=2$, we define the logarithmic characteristic cycle without the assumption on the cleanness of  $(X,U,\chi)$. By \cite[Theorem 4.1]{Ka94}, there exist successive blowups $\map{f}{X'}{X}$ at closed points such that $f^{-1}(U)$ is isomorphic to $U$
via $f$ and $(X',f^{-1}(U), f^*\chi)$ is clean. Let $D'$ be the inverse image of~$D$, and let
\[
T^*X(\log D)|_{D_F}\xleftarrow{\mathrm{pr}}T^*X(\log D)|_{D_F}\times_{D_F}{D'_F}\xrightarrow{df^D}T^*X'(\log D')|_{D'_F}
\]
be the algebraic correspondence. We define $\CC^{\log}j_!\calF-[T^*_XX(\log D)|_{D_F}]$ to be the pushforward by pr of the pullback of $\CC^{\log}j'_!f^*\calF-[T^*_{X'}{X'}(\log D')|_{D'_F}]$ by $df^D$. This is independent of the choice of blowups by \cite[Remark 5.7]{Ka94}. We define the logarithmic characteristic cycle as
\begin{equation}\label{5.15}
\CC^{\log}j_!\calF=\left[T^*_XX(\log D)|_{D_F}\right]+\sum_{i\in I_{W,\chi}}\sw\left(\chi|_{K_i}\right)\left[L_{i,\chi}\right]+\sum_{x\in D_F}s_x[T_x^*X(\log D)], 
\end{equation}
where $T_x^*X(\log D)$ denotes the fiber at $x$.

\begin{thm}[Conductor formula, \textit{cf.} {\cite[Corollary 7.5.3 and Theorem 8.3.7]{KS13}}]\label{5.1}
Assume $\dim X=2$ and $X$ is proper over $\calO_K$. Then we have 
\begin{equation*}
\left(\CC^{\log}j_!\calF-\left[T^*_XX(\log D)|_{D_F}\right], T^*_XX(\log D)|_{D_F}\right)_{T^*X(\log D)|_{D_F}}=-\Sw_K \left(X_{\overline{K}}, j_!\calF\right)+\Sw_K\left(X_{\overline{K}}, j_!\Lambda\right), 
\end{equation*}
where $\Sw_K(X_{\overline{K}}, j_!\calF)$ denotes the alternating sum $\sum_{m\ge 0} (-1)^m\Sw_K\rmH^m(X_{\overline{K}}, j_!\calF)$. 
\end{thm}

\begin{proof}
There exist successive blowups $\map{f}{X'}{X}$ at closed points such that $(X', f^{-1}(U), f^*\chi)$ is clean. Since  both sides do not change after replacing $X$ by $X'$, we may assume $(X,U,\chi)$ is clean. The intersection 
\[\left(
\CC^{\log}j_!\calF-[T^*_XX(\log D)|_{D_F}], T^*_XX(\log D)|_{D_F}\right)_{T^*X(\log D)|_{D_F}}
\]
equals $-\deg c_\chi$ with the notation in \cite[Equation~(8.3.0.1)]{KS13}. Hence the assertion follows by \cite[Corollary~7.5.3 and Theorem 8.3.7]{KS13}.
\end{proof}

We define the sub--vector bundle $L'_{i,\chi}$ of $FT^*X|_{D_i}$ for $i\in I_{W,\chi}$ as the image of the multiplication by the characteristic form of $\chi$, 
\[
\times \ch(\chi)|_{D_i}\colon\calO_X\left(-pR'_{\chi}\right)\otimes_{\calO_X}\calO_{D_i}\left(\sum_{x\in D_i}p\ord'(\chi;x,D_i)[x]\right)\lra F\Omega^1_X|_{D_i}. 
\]
For $i\in I_{T,\chi}$, we define $L'_{i,\chi}$ to be $\rmF^*(T^*_{D_i}X|_{D_{i,F}})$, where $\rmF^*$ is the pullback by the Frobenius $\map{\rmF}{D_{i,F}}{D_{i,F}}$.

\begin{definition}\label{5.17}
Assume $\dim X=2$. We define the \textup{F-characteristic cycle} $\FCC j_!\calF$ as a cycle on the FW-cotangent bundle $FT^*X|_{X_F}$ by 
\begin{equation}\label{5.2}
\FCC j_!\calF=-\left(\frac{1}{p}\left[FT_X^*X|_{X_F}\right]+\sum_{i\in I}\dt(\chi|_{K_i})\left[L'_{i,\chi}\right]+\sum_{x\in D_F}pt_x\left[\rmF^*T_x^*X\right]\right), 
\end{equation}
where $FT_X^*X|_{X_F}$ denotes the \textup{0}-section of $FT^*X|_{X_F}$ and 
\begin{equation}\label{5.10}
t_x=\#I_x-1+s_x+\sum_{i\in I_{W,x}}\sw\left(\chi|_{K_i}\right)(\ord'(\chi;x,D_i)-\ord(\chi;x,D_i))
+\sum_{i\in I_{\II,x}}(\ord(\chi;x,D_i)+1-\#I_x).
\end{equation}
Here, the integer $s_x$ is the coefficient of the fiber at $x$ in $\CC^{\log}j_!\calF$; see \eqref{5.15}. 
\end{definition}

The integrality of the characteristic form (Theorem~\ref{2.1}) is necessary to define $t_x$ for all closed points $x\in D_F$. 

\begin{lemma}  
Let $\map{h}{W}{X}$ be an \'etale morphism, and let $\map{j'}{W\times_XU}{W}$ be the base change of $j$. Then we have
\[
\FCC j'_!h^*\calF=h^*\FCC j_!\calF. 
\]
\end{lemma}

\begin{proof}
Let $K'_i$ be the local field at the generic point of $h^*D_i$. Then we have $\sw(\chi|_{K_i})=\sw((h^*\chi)|_{K'_i})$, $\dt(\chi|_{K_i})=\dt((h^*\chi)|_{K'_i})$,  $h^*\rsw(\chi)=\rsw(h^*\chi)$ and $h^*\ch(\chi)=\ch(h^*\chi)$. 
Hence the assertion follows from Definition~\ref{5.17}. 
\end{proof}

\begin{remark}
The F-characteristic cycle $\FCC j_!\calF$ is equal to
\[
-\left(\frac{1}{p}\left[FT_X^*X|_{X_F}\right]+\sum_{i\in J}\left[\rmF^*\left(T^*_{D_i}X|_{D_{i,F}}\right)\right]+\sum_{i\in J'}\dt(\chi|_{K_i})\left[L'_{i,\chi}\right]+\sum_{x\in D_F}pt_x\left[\rmF^*T_x^*X\right]\right), 
\]
where $J$ denotes the subset of $I$ consisting of the $i\in I$ such that $D_i\cap X_K$ is not empty and $J'$ denotes $I-J$.  Then $(1/p)[FT_X^*X|_{X_F}]+\sum_{i\in J}[\rmF^*(T^*_{D_i}X|_{D_{i,F}})]$ is a 1-cycle, and $\sum_{i\in J'}\dt(\chi|_{K_i})[L'_{i,\chi}]+\sum_{x\in D_F}pt_x[\rmF^*T_x^*X]$ is a 2-cycle. Later, we consider the difference $\FCC j_!\calF-\FCC j_!\Lambda$. This  is a 2-cycle, so the intersection number with the 0-section is defined. 
\end{remark}

\begin{remark}
In this remark, we consider the equal-characteristic case. 
Let $X$ be a smooth scheme over a perfect field $k$ of characteristic $p>0$. For simplicity, we assume $p\ne 2$. Let $D=\cup_{i\in I}D_i$ be a divisor with simple normal crossings, and let $\map{j}{U=X-D}{X}$ be the open immersion. Let $\calF$ be a locally constant sheaf of $\Lambda$-modules of rank~$1$ on $U$. Then the characteristic cycle $CCj_!\calF$ is defined as a cycle on the cotangent bundle $T^*X$. 
If the dimension of $X$ is~$2$, we have 
\[
CCj_!\calF=\left[T_X^*X\right]+\sum_{i\in I}\dt(\chi|_{K_i})\left[L''_{i,\chi}\right]+\sum_{x\in D_F}t_x[T_x^*X]
\]
by \cite[Theorem 6.1]{Ya20}. Here, $L''_{i,\chi}$ denotes the vector bundle defined by the characteristic form in the sense of \cite{Ya20}, and $t_x$ is defined in \cite{Ya20} by the same form as in~\eqref{5.10}.

Let $\map{\rmF}{X}{X}$ be the Frobenius. If we put 
\[
\FCC j_!\calF=-\left(\frac{1}{p}\left[\rmF^*T_X^*X\right]+\sum_{i\in I}\dt(\chi|_{K_i})\left[\rmF^*L''_{i,\chi}\right]+\sum_{x\in D_F}pt_x[\rmF^*T_x^*X]\right)
\]
as a cycle on $FT^*X\cong \rmF^*T^*X$, then we have 
\[
\rmF_*\FCC j_!\calF=-p\cdot CCj_!\calF, 
\]
where $\rmF_*$ denotes the pushforward by the projection $\rmF^*T^*X\to T^*X$. 
\end{remark}

The rationality of the characteristic form (Theorem~\ref{2.2}) implies the integrality of the coefficients of the fibers in the F-characteristic cycle.

\begin{lemma}
The coefficients $pt_x$ of the fibers $[\rmF^*T^*_xX]$ in the F-characteristic cycle \eqref{5.2} are integers. If\, $(X,U,\chi)$ is clean at $x\in D_F$, we have $t_x\ge 0$. 
\end{lemma}

\begin{proof}
In the definition \eqref{5.10} of $t_x$, the terms other than $\sw(\chi|_{K_i})\ord'(\chi;x,D_i)$ are integers. By Definition~\ref{5.18}, we see that the products $p\cdot \ord'(\chi;x,D_i)$ are integers.

If $(X,U,\chi)$ is clean at $x$, we have 
\[
t_x=\#I_x-1+\sum_{i\in I_{W,x}}\sw\left(\chi|_{K_i}\right)\cdot\ord'(\chi;x,D_i)
+\sum_{i\in I_{\II,x}}(1-\#I_x)
\]
by \eqref{5.10}. Since we have $\ord'(\chi;x,D_i)\ge 0$, we have $t_x\ge 0$ unless $\ord'(\chi;x,D_i)=0$ for all $i\in I_{W,\chi}$ and $\#I_x=\#I_{\II,x}=2$. If $\#I_x=\#I_{\II,x}=2$, we have $\rsw(\chi)_x=0$ by Proposition~\ref{3.3}, and this contradicts the assumption. 
\end{proof}

\begin{remark}
  The author conjectures that the terms $\sw(\chi|_{K_i})\ord'(\chi;x,D_i)$ are also integers and thus the $t_x$ are integers. We can check that $\sw(\chi|_{K_i})\ord'(\chi;x,D_i)$ is an integer in the following cases:
  \begin{enumerate}
    \item\label{case1} The character $\chi|_{K_i}$ is of type~\ref{I}.
    \item\label{case2} The character $\chi|_{K_i}$ is defined by a Kummer equation of degree $p$.
      \end{enumerate}
Indeed, in case~\eqref{case1}, $\ord'(\chi;x, D_i)$ is an integer. In  case~\eqref{case2}, if the character $\chi|_{K_i}$ is of type \ref{II}, the Swan conductor $\sw(\chi|_{K_i})$ is divisible by $p$. 

The author also conjectures that we have $t_x\ge 0$ even if $(X,U,\chi)$ is not clean at $x$. In the equal-characteristic case, this follows from the fact that $j_!\calF[2]$ is perverse by \cite[Proposition 5.14.1]{Sa17}. 
\end{remark}

Let $\map{\rmF}{X_F}{X_F}$ be the Frobenius. We define 
\[
\tau_D\colon F\Omega^1_X\lra \rmF^*\Omega^1_X(\log D)|_{X_F}
\]
as the composition of the maps 
\[
F\Omega^1_X\lra F\Omega^1_X/\calO_{X_F}\cdot w(p)\cong \rmF^*\Omega^1_{X_F}\lra \rmF^*\Omega^1_X(\log D)|_{X_F}, 
\]
where the middle isomorphism is the map \cite[Equation (4-1)]{Sa22-1}. 
The map $\tau_D$ defines a morphism 
\[
\tau_{D}\colon FT^*X|_{D_F}\lra\rmF^*T^*X(\log D)|_{D_F}
\]
of vector bundles over $D_F$. 

Let 
\[
\tau_{D}^!\colon\CH_2\left(FT^*X|_{D_F}\right)\lra\CH_2\left(T^*X(\log D)|_{D_F}\right)
\]
be the Gysin homomorphism for $\tau_D$.

\begin{lemma}[\textit{cf.} {\cite[Lemma 4.3(i)]{Ya20}}]\label{5.3}
Assume $\dim X=2$. Let $i$ be an element of\, $I_{\upI,\chi}$. Let $\map{\rmF}{D_i}{D_i}$ be the Frobenius.  
\begin{enumerate}
\item\label{5.3-1} We have $\dim \tau_D^{-1}(\rmF^*L_{i,\chi})=2$. 
\item\label{5.3-2} We have $L'_{i,\chi}=\rmF^*T^*_{D_i}X$. 
\item We have
\[
\tau_D^!\left(\left[\rmF^*L_{i,\chi}\right]\right)=\left[L'_{i,\chi}\right]+\sum_{x\in D_i}p\left(\ord'(\chi;x,D_i)-\ord(\chi;x,D_i)\right)[\rmF^*T^*_xX]+\sum_{x\in Z_{\II,\chi}\cap D_i}p[\rmF^*T^*_xX]
\]
 in $Z_2(\tau_D^{-1}(\rmF^*L_{i,\chi}))$.
\end{enumerate}
\end{lemma}

\begin{proof}
We may assume $I=\{1,2\}$ and $i=1$. Let $x$ be a closed point of $D_i$, and let $(\pi_1,\pi_2)$ be a local coordinate at $x$ such that $\pi_{i'}$ is a local equation of $D_i$ for $i'\in I_x$. Then $F\Omega^1_{X,x}$ is a free $O_{X,x}$-module with basis $(w\pi_1, w\pi_2)$. Its dual basis is denoted by $(\partial'/\partial'\pi_1,\partial'/\partial'\pi_2)$. 
Let $\mathcal{I},\mathcal{I}',\mathcal{J}$ be the defining ideal sheaves of $L_{1,\chi}\subset T^*X(\log D)|_{D_i}$, $L'_{i,\chi}\subset FT^*X|_{D_i}$ and $\tau_D^{-1}(\rmF^*L_{i,\chi})\subset FT^*X|_{D_i}$,  respectively. 

First, we consider the case $I_x=\{1\}$. Then $\Omega^1_X(\log D)_x$ is a free $O_{X,x}$-module with basis $(d\log \pi_1,d\pi_2)$. Its dual basis is denoted by $(\partial/\partial\log\pi_1,\partial/\partial\pi_2)$. If we put
\[
\rsw(\chi)_x=(\alpha_1d\log \pi_1+\alpha_2d\pi_2)/\pi_1^{n_1}, 
\]
where $\alpha_1,\alpha_2\in \calO_{X,x}$ and $n_1=\sw(\chi|_{K_1})$, then
\[
\mathcal{I}_x=\left(\pi_2^{-\ord(\chi;x,D_1)}\left(\alpha_2\partial/\partial\log\pi_1-\alpha_1\partial/\partial\pi_2\right)\right) 
\]
and 
\[
\mathcal{J}_x=\left(\pi_2^{-p\ord(\chi;x,D_1)}\left(-\alpha_1^p\partial/\partial\pi_2\right)\right).
\]
Since $\chi|_{K_1}$ is of type~\ref{I}, we have  
\[
\ch(\chi)_x=\alpha_1^p w\pi_1/\pi_1^{p(n_1+1)}
\]
by Proposition~\ref{3.2}, and thus we have $\ord'(\chi;x,D_i)=\ord_{\pi_2}(\alpha_1)$. Hence we have
\[
\mathcal{J}_x=\left(\pi_2^{p\ord'(\chi;x,D_1)-p\ord(\chi;x,D_1)}\partial/\partial\pi_2\right),
\]
\[
\mathcal{I}'_x=\left(\pi_2^{-p\ord'(\chi;x,D_1)}\left(-\alpha_1^p\partial'/\partial'\pi_2\right)\right)=(\partial'/\partial'\pi_2).
\]
The assertion follows.

Second, we consider the case $I_x=\{1,2\}$. Then $\Omega^1_X(\log D)_x$ is a free $O_{X,x}$-module with basis $(d\log \pi_1,d\log \pi_2)$. Its dual basis is denoted by $(\partial/\partial\log\pi_1,\partial/\partial\log\pi_2)$. If we put
\[
\rsw(\chi)_x=(\alpha_1d\log \pi_1+\alpha_2d\log\pi_2)/\pi_1^{n_1}\pi_2^{n_2}, 
\]
where $\alpha_1,\alpha_2\in \calO_{X,x}$ and $n_i=\sw(\chi|_{K_i})$, then
\[
\mathcal{I}_x=\left(\pi_2^{-\ord(\chi;x,D_1)}\left(\alpha_2\partial/\partial\log\pi_1-\alpha_1\partial/\partial\log\pi_2\right)\right),
\]
and
\[
\mathcal{J}_x=\left(\pi_2^{-p\ord(\chi;x,D_1)}\left(-\alpha_1^p\pi_2^p\partial/\partial\pi_2\right)\right).
\]
Since $i=1$ is an element of $I_{\textup{I},\chi}$, we have  
\[
\ch(\chi)|_{D_i,x}=\alpha_1^p\pi_2^{p\delta} w\pi_1/\pi_1^{p(n_1+1)}\pi_2^{p(n_2+\delta)}
\]
by Proposition~\ref{3.2}, where $\delta$ is 1 if $\chi|_{K_2}$ is tamely ramified or of type~\ref{I} and 0 if $\chi|_{K_2}$ is of type \ref{II}. Hence we have $\ord'(\chi;x,D_i)=\ord_{\pi_2}(\alpha_1)+\delta$ and 
\[
\mathcal{J}_x=\left(\pi_2^{p\ord'(\chi;x,D_1)-p\ord(\chi;x,D_1)+p(1-\delta)}\partial/\partial\pi_2\right),
\]
\[
\mathcal{I}'_x=(\partial'/\partial'\pi_2).
\]
The assertion follows. 
\end{proof}

\begin{lemma}[\textit{cf.} {\cite[Lemmas 4.4 and~4.5]{Ya20}}]\label{5.4}
Assume $\dim X=2$. Let $i$ be an element of $I_{\II,\chi}$, and let $\map{q'_i}{\tau_D^{-1}(\rmF^*L_{i,\chi})}{D_i}$ be the canonical projection. Let $\map{\rmF}{D_i}{D_i}$ be the Frobenius.  
\begin{enumerate}
\item\label{5.4-1} We have $\tau_D^{-1}(\rmF^*L_{i,\chi})=FT^*X|_{D_i}$. 
\item\label{5.4-2} We have
\[
\tau_D^!\left(\left[\rmF^*L_{i,\chi}\right]\right)={q'_i}^*\left(c_1\left(\rmF^*T^*X\left(\log D\right)|_{D_i}\right)\cap [D_i]-c_1\left(\calO_X\left(-pR_{\chi}\right)|_{D_i}\right)\cap [D_i]-\sum_{x\in D_i}p\ord(\chi;x,D_i)[x]\right)
\]
in $\CH_2(FT^*X|_{D_i})$. 
\item\label{5.4-3} We have 
\[
\left[L'_{i,\chi}\right]={q'_i}^*\left(c_1(FT^*X|_{D_i})\cap [D_i]-c_1\left(\calO_X\left(-pR'_{\chi}\right)|_{D_i}\right)\cap [D_i]-\sum_{x\in D_i}p\ord'(\chi;x,D_i)[x]\right)
\]
in $\CH_2(FT^*X|_{D_i})$.
\item\label{5.4-4} We have 
\[
\left[\rmF^*T^*_{D_i}X\right]={q'_i}^*\left(c_1(\calO_X\left(-pR_{\chi})|_{D_i}\right)\cap [D_i]+\sum_{x\in D_i}p(\ord(\chi;x,D_i)-\#I_x+1)[x]\right)
\]
in $\CH_2(FT^*X|_{D_i})$.
\end{enumerate}
\end{lemma}

\begin{proof}
\eqref{5.4-1}~  We use the same notation as in the proof of Lemma~\ref{5.3}. Since $\chi|_{K_1}$ is of type~\ref{II}, the refined Swan conductor is the image of the characteristic form by Proposition~\ref{3.3} and thus $\alpha_1=0$. Hence $\mathcal{I}'=0$ and we have
$\tau_D^{-1}(\rmF^*L_{i,\chi})=FT^*X|_{D_i}$.

\eqref{5.4-2}~ By applying the excess intersection formula to the cartesian diagram
\[
   \xymatrix{
    D_i \ar[r] \ar@{=}[d] &  \ar[d]  \rmF^*L_{i,\chi} \\
   D_i\ar[r] & \rmF^*T^*X(\log D)|_{D_i}\ar@{}[lu]|{\square} ,
   }
\]
we see that 
\[
\left[\rmF^*L_{i,\chi}\right]=q_i^*\left(c_1(\rmF^*T^*X(\log D)|_{D_i})\cap [D_i]-c_1\left(\rmF^*L_{i,\chi}\right)\cap [D_i]\right)
\]
in $\CH_2(\rmF^*T^*X(\log D)|_{D_i})$, where the map $\map{q_i}{\rmF^*T^*X(\log D)|_{D_i}}{D_i}$ is the canonical projection. 
Since the sub--vector bundle $L_{i,\chi}$ of $T^*X(\log D)|_{D_i}$ is defined by the image of the injection
\[
\times \rsw(\chi)|_{D_i}\colon\calO_X(-R_{\chi})\otimes_{\calO_X}\calO_{D_i}\left(\sum_{x\in D_i}\ord(\chi;x,D_i)[x]\right)\lra\Omega^1_X(\log D)|_{D_i}, 
\]
the assertion holds.

\eqref{5.4-3}~ By applying the excess intersection formula to the cartesian diagram
\[
   \xymatrix{
    D_i \ar[r] \ar@{=}[d] &  \ar[d]  L'_{i,\chi} \\
   D_i\ar[r] & FT^*X|_{D_i}\ar@{}[lu]|{\square},
   }
\]
we see that 
\[
\left[L'_{i,\chi}\right]={q'_i}^*\left(c_1\left(FT^*X|_{D_i}\right)\cap [D_i]-c_1\left(L'_{i,\chi}\right)\cap [D_i]\right)
\]
in $\CH_2(FT^*X|_{D_i})$. 
Since the sub--vector bundle $L'_{i,\chi}$ of $FT^*X|_{D_i}$ is defined by the image of the injection
\[
\times \ch(\chi)|_{D_i}\colon\calO_X\left(-pR'_{\chi}\right)\otimes_{\calO_X}\calO_{D_i}\left(\sum_{x\in D_i}p\ord'(\chi;x,D_i)[x]\right)\lra F\Omega^1_X|_{D_i},
\]
the assertion holds.

\eqref{5.4-4}~ By applying the excess intersection formula to the cartesian diagram
\[
   \xymatrix{
    D_i \ar[r] \ar@{=}[d] &  \ar[d]  \rmF^*T^*_{D_i}X \\
   D_i\ar[r] & FT^*X|_{D_i}\ar@{}[lu]|{\square},
   }
\]
we see that 
\[
\left[\rmF^*T^*_{D_i}X\right]={q'_i}^*(c_1(\rmF^*T^*D_i)\cap [D_i])
\]
in $\CH_2(FT^*X|_{D_i})$ since the sequence 
\[
0\lra \rmF^*T^*_{D_i}X\lra FT^*X|_{D_i}\lra FT^*D_i\lra 0
\]
is exact by \cite[Equation~(2-12)]{Sa22-2} and $FT^*D_i\cong \rmF^*T^*D_i$ by \cite[Equation~(2-4)]{Sa22-2}. Since $D_i$ is a scheme over $F$, the computation in \cite[Lemma 4.5(iii)]{Ya20} implies that we have an equality
\[
c_1(T^*D_i)\cap [D_i]=c_1\left(\calO_{D_i}\left(-\left(R_{\chi}\cap D_i\right)\right)\cap [D_i]+\sum_{x\in D_i}\left(\ord(\chi;x,D_i\right)-\#I_x+1\right)[x]
\]
in $\CH_0(D_i)$. 
Since $\rmF^*\Omega^1_{D_i}=(\Omega^1_{D_i})^{\otimes p}$, we have
\[
c_1(\rmF^*T^*D_i)\cap [D_i]=c_1\left(\calO_X\left(-pR_{\chi}\right)|_{D_i})\cap [D_i]+\sum_{x\in D_i}p\left(\ord(\chi;x,D_i\right)-\#I_x+1\right)[x]
\]
in $\CH_0(D_i)$. 
\end{proof}

\begin{lemma}\label{5.5}
Assume $D_i$ is contained in the closed fiber $D_F$. Let $\map{\rmF}{D_i}{D_i}$ be the Frobenius. Then we have 
\[
c_1\left(\rmF^*T^*X(\log D)|_{D_i}\right)\cap [D_i]=c_1\left(FT^*X|_{D_i}\right)\cap [D_i]+\sum_{j\in I}c_1\left(\calO_{X}(pD_j)|_{D_i}\right)\cap [D_i] 
\]
in $\CH_0(D_F)$. 
\end{lemma}

\begin{proof}
Let $D'_i$ be the closed subscheme consisting of the closed points $x$ of $D_i$ such that $\#I_x=2$. 
By the two exact sequences
\[
0\lra \rmF^*N_{D_i/X}\lra F\Omega^1_X|_{D_i}\lra  \rmF^*\Omega^1_{D_i}\lra 0,
\]
\[
0\lra \rmF^*\Omega^1_{D_i}(\log D'_i)\lra\rmF^*\Omega^1_{X}(\log D)|_{D_i}\xrightarrow{\rmF^*\textrm{res}} \calO_{D_i}\lra 0
\]
of locally free $D_i$-modules, we have 
\begin{multline*}
c_1\left(\rmF^*T^*X(\log D)|_{D_i}\right)\cap[D_i]-c_1\left(FT^*X|_{D_F}\right)\cap[D_i]=\\
c_1\left(\rmF^*\Omega^1_{D_i}(\log D'_i)\right)\cap[D_i]-c_1\left(\rmF^*\Omega^1_{D_i}\right)\cap[D_i]-c_1\left(\rmF^*N_{D_i/X}\right)\cap[D_i].
\end{multline*}
Applying \cite[Equation~(4.10)]{Ya20} to the scheme $D_i$, we obtain
\[
c_1\left(\Omega^1_{D_i}(\log D'_i)\right)\cap[D_i]-c_1\left(\Omega^1_{D_i}\right)\cap[D_i]=\sum_{\substack{j\in I \\ j\ne i}}c_1\left(\calO_{D_i}(D_j\cap D_i)\right)\cap [D_i].
\]
Since we have $\rmF^*N_{D_i/X}=\rmF^*\calO_X(-D_i)|_{D_i}$, the assertion follows. 
\end{proof}

\begin{thm}\label{main}
Assume $\dim X=2$ and $X$ is proper over $\calO_K$. Then we have 
\[
\left(\FCC j_!\calF-\FCC j_!\Lambda, FT^*_XX|_{X_F}\right)_{FT^*X|_{X_F}}=p\cdot\left(\Sw_K\left(X_{\overline{K}}, j_!\calF\right)-\Sw_K\left(X_{\overline{K}}, j_!\Lambda\right)\right).
\]
\end{thm}

\begin{proof}
We do some computations used later. By Lemma~\ref{5.3}, we have 
\begin{multline}\label{5.7}
\tau_D^!\left(\sum_{i\in I_{\textup{I},\chi}}\sw(\chi|_{K_i})\left[\rmF^*L_{i,\chi}\right]\right)=\\
\sum_{i\in I_{\textup{I},\chi}}\sw\left(\chi|_{K_i}\right)\left(\left[L'_{i,\chi}\right]
+\sum_{x\in D_i}p(\ord'(\chi;x,D_i)-\ord(\chi;x,D_i))[\rmF^*T^*_xX]+\sum_{x\in Z_{\II,\chi}\cap D_i}p[\rmF^*T^*_xX]\right).
\end{multline}
We note that if $i\in I_{\textup{I},\chi}$, we have
\[
{q'_i}^*\left(c_1\left(\calO_X\left(Z_{\II,\chi}\right)|_{D_i}\right)\cap [D_i]\right)=\sum_{x\in Z_{\II,\chi}\cap D_i}[\rmF^*T^*_xX], 
\]
where $q'_i$ denotes the canonical projection $\map{q'_i}{\tau_D^{-1}(\rmF^*L_{i,\chi})}{D_i}$ as in Lemma~\ref{5.4}.

By Lemma~\ref{5.4}\eqref{5.4-2} and~\eqref{5.4-3} and Lemma~\ref{5.5}, we have 
\begin{multline}\label{5.8}
  \tau_D^!\left(\sum_{i\in I_{\II,\chi}}\sw\left(\chi|_{K_i}\right)\left[\rmF^*L_{i,\chi}\right]\right)=\\
  \sum_{i\in I_{\II,\chi}}\sw\left(\chi|_{K_i}\right)\left(\left[L'_{i,\chi}\right]
  +\sum_{x\in D_i}p\left(\ord'(\chi;x,D_i)-\ord(\chi;x,D_i)\right)[\rmF^*T^*_xX]+{q'_i}^*\left(c_1\left(\calO_X\left(pZ_{\II,\chi}\right)|_{D_i}\right)\cap [D_i]\right)\right).
\end{multline}
Since we have
\[
\sum_{i\in I}\sw\left(\chi|_{K_i}\right){q'_i}^*\left(c_1\left(\calO_X\left(pZ_{\II,\chi}\right)|_{D_i}\right)\cap[D_i]\right)=\sum_{i\in I_{\II,\chi}}{q'_i}^*\left(c_1\left(\calO_X\left(pR_{\chi}\right)|_{D_i}\right)\cap [D_i]\right), 
\]
we have 
\begin{multline}
\sum_{i\in I}\sw\left(\chi|_{K_i}\right){q'_i}^*\left(c_1\left(\calO_X\left(pZ_{\II,\chi}\right)|_{D_i}\right)\cap[D_i]\right)=\\
\sum_{i\in I_{\II, \chi}}\left(-[\rmF^*T^*_{D_i}X]+\sum_{x\in D_i}p(\ord(\chi;x,D_i)-\#I_x+1)[\rmF^*T_x^*X])\right)
\end{multline}
by Lemma~\ref{5.4}\eqref{5.4-4}. The sum of equalities (\ref{5.7}) and (\ref{5.8}) gives the equality
\begin{multline}\label{5.9}
\tau_D^!\left(\sum_{i\in I}\sw\left(\chi|_{K_i}\right)\left[\rmF^*L_{i,\chi}\right]\right)=\\
-\sum_{i\in I_{\II,\chi}}\left[\rmF^*T^*_{D_i}X\right]+
\sum_{i\in I}\sw\left(\chi|_{K_i}\right)\left(\left[L'_{i,\chi}\right]+
\sum_{x\in D_i}p\left(\ord'(\chi;x,D_i)-\ord(\chi;x,D_i)\right)[\rmF^*T^*_xX]\right)\\
+\sum_{i\in I_{\II,\chi}}\sum_{x\in D_i}p(\ord(\chi;x,D_i)-\#I_x+1)[\rmF^*T_x^*X].
\end{multline}

We have 
\begin{equation}\label{5.14}
\begin{split}
\left(\rmF^*L_{i,\chi}, T^*_XX(\log D)|_{D_F}\right)_{\rmF^*T^*X(\log D)|_{D_F}}&=c_1\left(\rmF^*\left(T^*X(\log D)|_{D_F}/L_{i,\chi}\right)\right)\cap[D_F]\\
&=p\cdot\left(c_1\left(T^*X(\log D)|_{D_F}/L_{i,\chi}\right)\cap[D_F]\right)\\
&=p\cdot\left(L_{i,\chi}, T^*_XX(\log D)|_{D_F}\right)_{T^*X(\log D)|_{D_F}}.
\end{split}
\end{equation}

First, we assume that $s_x=0$ for every $x\in D_F$. Then it suffices to show that we have 
\begin{equation}\label{5.6}
\FCC j_!\calF-\FCC j_!\Lambda=-\tau_D^!\left(\rmF^*\left(\CC^{\log}j_!\calF-\left[T^*_XX(\log D)|_{D_F}\right]\right)\right)
\end{equation}
in $\CH_2(FT^*X|_{D_F})$ by Theorem~\ref{5.1} and \eqref{5.14}. 
This equality holds by \eqref{5.9} and Lemma~\ref{5.3}\eqref{5.3-2} and the definition  (\ref{5.10}) of $t_x$. 

Next, we consider the general case. By the definition of the logarithmic characteristic cycle, we have
\[
\CC^{\log}j_!\calF-\left[T^*_XX(\log D)|_{D_F}\right]-\sum_{x\in D_F}s_x\left[T_x^*X(\log D)\right]=\sum_{i\in I}\sw\left(\chi|_{K_i}\right)\left[L_{i,\chi}\right].
\]
Then the equality (\ref{5.9}) shows that we have
\begin{equation}\label{5.12}
\begin{split}
&\tau_D^!\left(-\rmF^*\left(\CC^{\log}j_!\calF-\left[T^*_XX(\log D)|_{D_F}\right]-\sum_{x\in D_F}s_x\left[T^*_xX(\log D)\right]\right)\right)\\
&= \tau_D^!\left(-\sum_{i\in I}\sw\left(\chi|_{K_i}\right)\rmF^*\left[L_{i,\chi}\right]\right)\\
&= \FCC j_!\calF-\FCC j_!\Lambda+\sum_{x\in D_F}ps_x[\rmF^*T^*_xX]
\end{split}
\end{equation}
by the definition of $t_x$. 
By Theorem~\ref{5.1}, we have 
\begin{multline}
\left(-\left(\CC^{\log}j_!\calF-\left[T^*_XX(\log D)|_{D_F}\right]-\sum_{x\in D_F}s_x\left[T_x^*X(\log D)\right]\right), T^*_XX(\log D)|_{D_F}\right)_{T^*X(\log D)|_{D_F}}=\\
\Sw_K \left(X_{\overline{K}}, j_!\calF\right)-\Sw_K\left(X_{\overline{K}}, j_!\Lambda\right)+\sum_{x\in D_F}s_x
\end{multline}
since the intersection number $(T_x^*X(\log D), T^*_XX(\log D)|_{D_F})$ is 1. By (\ref{5.14}), we have 
\begin{multline}\label{5.13}
\left(-\rmF^*\left(\CC^{\log}j_!\calF-\left[T^*_XX(\log D)|_{D_F}\right]-\sum_{x\in D_F}s_x\left[T_x^*X(\log D)\right]\right), T^*_XX(\log D)|_{D_F}\right)_{\rmF^*T^*X(\log D)|_{D_F}}=\\
p\cdot\left(\Sw_K \left(X_{\overline{K}}, j_!\calF\right)-\Sw_K\left(X_{\overline{K}}, j_!\Lambda\right)+\sum_{x\in D_F}s_x\right).
\end{multline}
By (\ref{5.12}) and (\ref{5.13}), we have
\begin{multline*}
\left(\FCC j_!\calF-\FCC j_!\Lambda+\sum_{x\in D_F}ps_x[\rmF^*T_x^*X], FT^*_XX|_{D_F}\right)_{FT^*X|_{D_F}}=\\
p\cdot\left(\Sw_K \left(X_{\overline{K}}, j_!\calF\right)-\Sw_K\left(X_{\overline{K}}, j_!\Lambda\right)+\sum_{x\in D_F}s_x\right).
\end{multline*}
Since the intersection number 
$(\rmF^*T_x^*X, FT^*_XX|_{D_F})$ is $1$, 
the assertion follows. 
\end{proof}

\subsection{Example}
In this subsection, we give an example of the F-characteristic cycle.

Let $p>2$ be a prime number, and let $\zeta_p$ be a primitive $\supth{p}$ root of unity.  Let $K$ be a complete discrete valuation field tamely ramified over $\Q_p(\zeta_p)$ with valuation ring $\calO_K$ and with residue field $F$. Let $e=\ord_K p$ be the absolute ramification index, and put $e'=pe/(p-1)$.  We fix a uniformizer $\pi$ and write $p=u\pi^e$ with some $u\in \calO_K^{\times}$.  Let $a,b,c$ be integers satisfying $0<a,b<p$, $(p,c)=1$ and $a+b+c=0$. We put $X=\mathbf{P}^1_{\calO_K}$. Let $U$ be the open subscheme $\Spec K[x^{\pm 1}, (1-x)^{-1}]$ of $X$, and let $j\colon U \to X$ be the open immersion. Let $\calK$ be the Kummer sheaf defined by $t^p=(-1)^cx^a(1-x)^b$ on $U$. For convenience, we change the coordinate to $y=x+a/c$. Let $D, E_1, E_2, E_3$ be divisors defined by $(\pi=0), (y-a/c=0), (y+b/c=0), (y=\infty)$, respectively. Then $D\cup E_1\cup E_2\cup E_3$ is a divisor with simple normal crossings, and $U$ is the complement of this divisor. Let $z_0$ be the closed point
$\{\pi=y=0\}$, and let $z_i$ be the closed point $E_i\cap D$ for $i=1,2,3$. Let $M$ be the local field at the generic point of $D$.

We compute the F-characteristic cycle $\FCC j_!\calK$ of $j_!\calK$. Applying Theorem~\ref{main}, we compute the Swan conductor of $\rmH^1(\bfP^1_{\overline{K}}, j_!\calK)$. This cohomology group realizes the Jacobi sum Hecke character as in \cite{CM88}. Coleman--McCallum \cite{CM88}, Miki \cite{Mi94} and Tsushima \cite{Ts10} computed the conductor or, explicitly, the ramified component of the Jacobi sum Hecke character in more general cases by different methods.

\begin{remark}
 We note that the Swan conductor can be calculated more easily by computing the logarithmic characteristic cycle and applying Kato--Saito's conductor formula (Theorem~\ref{5.1}) because we need to compute the logarithmic characteristic cycle for the computation of the F-characteristic cycle. The subject of this article is  non-logarithmic theory, so we compute the Swan conductor using the F-characteristic cycle. 
 \end{remark}
 
We write $\chi$ for the character corresponding to $\mathcal{K}$. We have $\sw(\chi|_M)=\dt(\chi|_M)=e'$, and the character $\chi|_M$ is of type~\ref{II}. We have
\begin{equation}
\rsw\chi=\frac{-cy\cdot dy}{(1-\zeta_p)^{p}\cdot (y-a/c)(y+b/c)}
\end{equation}
and
\begin{equation}\label{5.3.4}
\ch\chi=\frac{-c^py^p\cdot wy}{(1-\zeta_p)^{p^2}\cdot (y-a/c)^{p}(y+b/c)^{p}}
\end{equation}
on the complement of $E_3$ by \cite[Corollary 8.2.3]{KS13}. 
The character $\chi$ is not clean and not non-degenerate only at $z_0$, and we have $\ord(\chi; z_0, D)=\ord'(\chi; z_0, D)=1$. 

We prove an elementary lemma used later.

\begin{lemma}\label{5.3.1}
Let $d$ be a rational number satisfying $v_p(d)\ge 0$. We put $r=v_p(d^{p-1}-1)$. Then there exists an integer $l$ such that $p^r$ divides $1-dl^p$. There does not exist any integer $l$ such that $p^{r+1}$ divides $1-dl^p$. 
\end{lemma}

\begin{proof}
Since $p^r$ divides $d^{p-1}-1=(d-1)(d^{p-2}+\cdots+ d+1)$, we see that $p^r$ divides $d-1$ or $d(d^{p-3}+\cdots +1)+1$. Therefore, we may take $l=1$ or $l=-(d^{p-3}+\cdots +1)$ because we have $dl^p\equiv d^pl^p\equiv 1 \mod p^r$. 

If there exists an integer $l$ such that $p^{r+1}$ divides $1-dl^p$, then we have $v_p(d^{p-1}-1)\ge r+1$, which gives a contradiction. 
\end{proof}

Now we compute the F-characteristic cycle of $j_!\calK$. We have to divide the computation into two cases.

\medskip
\begin{enumerate}[label=\textit{Case}~\arabic*., ref=\arabic*]
\item  We assume $v_p((a^ab^bc^c)^{p-1}-1)=1$. \label{case5-1}
\end{enumerate}

\begin{lemma}\label{5.3.2}
We put $Y_n=\Spec \calO_K[y_n]$ for a natural number $1\le n\le e/2$. Let $M_n$ be the local field at the generic point of the closed fiber. Let $\chi_n$ be the Kummer character defined around $(y_n=0)$ by the equation $t^p=(-1)^a(y_n\pi^n-a/c)^a(y_n\pi^n+b/c)^b$. Then we have the following properties:  
\begin{enumerate}
  \item\label{5.3.2-1} We have $\sw(\chi_n|_{M_n})=e'-2n$. 

  \item\label{5.3.2-2} If $n<e/2$, the character $\chi_n$ is not clean at $y_n=0$. 

  \item\label{5.3.2-3} If $n=e/2$, the character $\chi_n$ is clean. 
\end{enumerate}
\end{lemma}

\begin{proof}
By Lemma~\ref{5.3.1}, we may take an integer $l$ such that $p$ divides $1-a^ab^bc^cl^p$. Then we have 
\begin{align*}
(lt)^p&=(-1)^al^p(y_n\pi^n-a/c)^a(y_n\pi^n+b/c)^b\\
    &=l^p\left(a^ab^bc^c+\frac{a^{a-1}b^{b-1}}{2c^{a+b-3}}y_n^2\pi^{2n}+\cdots\right)\\
    &=1+pm+\frac{a^{a-1}b^{b-1}l^p}{2c^{a+b-3}}y_n^2\pi^{2n}+\cdots\\
    &=1+mu\pi^e+\frac{a^{a-1}b^{b-1}l^p}{2c^{a+b-3}}y_n^2\pi^{2n}+\cdots
      \end{align*}
  for some rational number $m$ such that $v_p(m)=0$, and the omitted part is divisible by $\pi^{3n}$.  Hence,  assertion~\eqref{5.3.2-1} follows. Around the closed point $\{\pi=y_n=0\}$, the refined Swan conductor is
\[
\rsw(\chi)=\frac{-a^{a-1}b^{b-1}c^{3-a-b}\left(n y_n^2\cdot d\log \pi+y_n\cdot dy_n\right)}{(1-\zeta_p)^p\cdot \pi^{-2n}\cdot \left(y_n\pi^n-a/c\right)\left(y_n\pi^n+b/c\right)}
\]
if $n<e/2$ and 
\[
\rsw(\chi)=\frac{-\left(\left(2^{-1}a^{a-1}b^{b-1}c^{3-a-b}l^py_n^2+mu\right)e\cdot d\log \pi +a^{a-1}b^{b-1}c^{3-a-b}l^p y_n\cdot dy_n\right)}{(1-\zeta_p)^p\cdot \pi^{-2n}\cdot \left(y_n\pi^n-a/c\right)\left(y_n\pi^n+b/c\right)l^p}
\]
if $n=e/2$ by \cite[Corollary 8.2.3]{KS13}. Since we assume $m$ is prime to $p$, assertions~\eqref{5.3.2-2} and~\eqref{5.3.2-3} follow.
\end{proof}

We now compute the coefficient $s_{z_0}$ of the fiber in the logarithmic characteristic cycle. We may work locally around $z_0$.  We define the successive blowups  as follows.

Let $X_1\to X$ be the blowup at the closed point $\{\pi=y=0\}$. The scheme $X_1$ is a union of two open subschemes $U_1=\Spec \calO_K[y, x_1]/(yx_1-\pi)$ and $Y_1=\Spec \calO_K[y_1]$, where $y_1\pi=y$. Then we can check that the character $\chi$ is clean on $U_1$. By Lemma~\ref{5.3.2}, $\chi$ is not clean at $y_1=0$. Let $X_2\to X_1$ be the blowup at the closed point $\{\pi=y_1=0\}$. Repeating this process, we get the successive blowups $X_{e/2}\to \cdots \to X_1\to X_0=X$ at non-clean closed points. The character $\chi$ is clean on $X_{e/2}$ by Lemma~\ref{5.3.2}. Hence the coefficient $s_{z_0}$ is equal to $e'-\sum_{1\le i\le e/2}2=e'-e$ by \cite[Remark 5.8]{Ka94}. Hence we have 
\[
\FCC j_!\calK-\FCC j_!\Lambda=-e'[L']+\left[\rmF^*T^*_{X_F}X\right]-p(e'-e+1)\left[\rmF^*T^*_{z_0}X\right]+p\left[\rmF^*T_{z_1}^*X\right]+p\left[\rmF^*T_{z_2}^*X\right]+p\left[\rmF^*T_{z_3}^*X\right], 
\]
where $L'$ is defined by the characteristic form \eqref{5.3.4}. 

We now compute the intersection number with the 0-section. 
We have 
\begin{align*}
\left(\left[\rmF^*T^*_{X_F}X\right], FT^*_XX|_{X_F}\right)_{FT^*X|_{X_F}}=c_1\left(\rmF^*\Omega^1_{\bfP^1_F}\right)\cap \left[\bfP^1_F\right]=-2p. 
\end{align*}
Since $L'$ is defined by the image of the injection
\[
\times\ch\chi\colon \calO_{X}(-p(e'D+E_1+E_2+E_3))\otimes_{\calO_X}\calO_D(p[z_0])\lra F\Omega^1_X|_D
,
\]
we have 
\begin{align*}
\left([L'], FT^*_XX|_{X_F}\right)_{FT^*X|_{X_F}}
&=c_1\left(F\Omega^1_X|_{X_F}\right)\cap [X_F]+pe'(D,D)_X+3p-p\\
&=c_1\left(\rmF^*\Omega^1_{X_F}\right)\cap [X_F]+c_1\left(\rmF^*N_{X_F}X\right)\cap [X_F]+2p\\
&=-2p+0+2p=0. 
\end{align*}

Hence we have 
\[
\left(\FCC j_!\calK-\FCC j_!\Lambda, FT^*_XX|_{X_F}\right)_{FT^*X|_{X_F}}=-p(e'-e). 
\]
Since we have $\rmH^i(X, j_!\calK)=0$ for $i\ne 1$ and $\Sw_K(X_{\overline{K}}, j_!\Lambda)=0$, the Swan conductor of $\rmH^1(\mathbf{P}^1_{\overline{K}}, j_!\calK)$ is $e'-e=e/(p-1)$ by Theorem~\ref{main}. 

\medskip
\begin{enumerate}[label=\textit{Case}~\arabic*., ref=\arabic*, resume]
\item We assume $v_p((a^ab^bc^c)^{p-1}-1)\ge 2$. \label{case5-2}
\end{enumerate}
 
\begin{lemma}\label{5.3.3}
  We put $Y_n=\Spec \calO_K[y_n]$ for a natural number $1\le n\le (e'-1)/2$. Let $\chi_n$ be the Kummer character defined around $(y_n=0)$ by the equations $t^p=(-1)^a(y_n\pi^n-a/c)^a(y_n\pi^n+b/c)^b$. Then we have the following claims:
  \begin{enumerate}
    \item We have $\sw(\chi_n|_{M_n})=e'-2n$. 
    \item The character $\chi_n$ is not clean at $y_n=0$.
      \end{enumerate}
\end{lemma}

\begin{proof}
We can prove the assertions in the same way as Lemma~\ref{5.3.2}. 
\end{proof}

We now compute the coefficient $s_{z_0}$ of the fiber in the logarithmic characteristic cycle. We may work locally around $z_0$.  Take the successive blowups $X_{(e'-1)/2}\to \cdots \to X_0=X$ at non-clean closed points in the same way as in Case~\ref{case5-1}. Unlike in Case~\ref{case5-1}, the character $\chi$ is still not clean on $X_{(e'-1)/2}$. 

The scheme $X_{(e'-1)/2}$ contains the open subscheme $Y_{(e'-1)/2}=\Spec\calO_K[y_{(e'-1)/2}]$. We put $y'=y_{(e'-1)/2}$. Let $W\to X_{(e'-1)/2}$ be the blowup at the closed point $\pi=y'=0$. The scheme $W$ is the union of two open subschemes $U=\Spec \calO_K[y', w]/(y'w-\pi)$ and $V=\Spec \calO_K[y'']$,  where $y''\pi=y'$. The character $\chi$ is unramified on $V$. 
On $U$, the character $\chi$ is defined by 
\[
t^p=(-1)^a\left(y'^{(e'+1)/2}w^{(e'-1)/2}-a/c\right)^a\left(y'^{(e'+1)/2}w^{(e'-1)/2}+b/c\right)^b. 
\]
We can check $\chi$ is not clean at the closed point $\{y'=w=0\}$. Further, let $W'\to W$ be the blowup at the closed point $\{y'=w=0\}$ and $U'$ be the open subscheme $\Spec \calO_K[y', w, w']/(y'w'-w, y'w-\pi)$. Then the the character $\chi$ is defined by 
\[
t^p=(-1)^a\left(y'^{e'}w'^{(e'-1)/2}-a/c\right)^a\left(y'^{e'}w'^{(e'-1)/2}+b/c\right)^b, 
\]
and the refined Swan conductor of $\chi$ is 
\begin{align*}
\rsw(\chi)
&=\frac{-2^{-1}a^{a-1}b^{b-1}c^{3-a-b}(e'-1) \cdot d\log w'}{(1-\zeta_p)^py'^{-2e'}w'^{-(e'-1)}\cdot \left(y'^{e'}w'^{(e'-1)/2}-a/c\right)\left(y'^{e'}w'^{(e'-1)/2}+b/c\right)}
\end{align*}
by \cite[Corollary 8.2.3]{KS13}. Hence the character $\chi$ is clean on $W'$. 

We see that the coefficient $s_{z_0}$ is equal to $e'-(\sum_{1\le i\le (e'-1)/2}2+1)=0$ by \cite[Remark 5.8]{Ka94}. Hence we have 
\[
\FCC j_!\calK-\FCC j_!\Lambda=-e'[L']+\left[\rmF^*T^*_{X_F}X\right]-p\left[\rmF^*T^*_{z_0}X\right]+p\left[\rmF^*T_{z_1}^*X\right]+p\left[\rmF^*T_{z_2}^*X\right]+p\left[\rmF^*T_{z_3}^*X\right].
\]
Computing  as in Case~\ref{case5-1}, we obtain  
\[
(\FCC j_!\calK-\FCC j_!\Lambda, FT^*_XX|_{X_F})_{FT^*X|_{X_F}}=0, 
\]
and the Swan conductor of $\rmH^1(\mathbf{P}^1_{\overline{K}}, j_!\calK)$ is 0 by Theorem~\ref{main}.


\end{document}